\RequirePackage{fixltx2e}
\documentclass[10pt]{article}
\usepackage[latin9]{inputenc}
\usepackage{color}
\usepackage{float}
\usepackage{mathrsfs}
\usepackage{amsmath}
\usepackage{amsthm}
\usepackage{amssymb}
\usepackage[a4paper]{geometry}
\usepackage[unicode=true,pdfusetitle,
 bookmarks=true,bookmarksnumbered=false,bookmarksopen=false,
 breaklinks=false,pdfborder={0 0 1},backref=false,colorlinks=true]
 {hyperref}

\makeatletter

\providecommand{\tabularnewline}{\\}

\theoremstyle{plain}
\newtheorem{thm}{\protect\theoremname}[section]
\theoremstyle{plain}
\newtheorem{cor}[thm]{\protect\corollaryname}
\theoremstyle{plain}
\newtheorem{prop}[thm]{\protect\propositionname}
\theoremstyle{plain}
\newtheorem{conjecture}[thm]{\protect\conjecturename}
\theoremstyle{remark}
\newtheorem{rem}[thm]{\protect\remarkname}
\theoremstyle{definition}
\newtheorem*{example*}{\protect\examplename}
\theoremstyle{definition}
\newtheorem{defn}[thm]{\protect\definitionname}
\theoremstyle{plain}
\newtheorem{lem}[thm]{\protect\lemmaname}
\theoremstyle{remark}
\newtheorem*{rem*}{\protect\remarkname}

\@ifundefined{date}{}{\date{}}
\usepackage{graphicx}
\usepackage{fancyhdr}\usepackage{setspace}\usepackage{mathrsfs}\usepackage{enumerate}\usepackage{tikz}\usepackage{fixltx2e}\usepackage{amsthm}\usepackage{ytableau}

\theoremstyle{definition}

\makeatother

\providecommand{\conjecturename}{Conjecture}
\providecommand{\corollaryname}{Corollary}
\providecommand{\definitionname}{Definition}
\providecommand{\examplename}{Example}
\providecommand{\lemmaname}{Lemma}
\providecommand{\propositionname}{Proposition}
\providecommand{\remarkname}{Remark}
\providecommand{\theoremname}{Theorem}

\begin{document}
\global\long\def\C{\mathbb{C}}%
\global\long\def\cnk{\left(\C^{n}\right)^{\otimes k}}%
\global\long\def\P{P_{k}(n)}%
\global\long\def\A{A_{k}(n)}%
\global\long\def\element{\left(e_{1}-e_{2}\right)\otimes\dots\otimes\left(e_{2k-1}-e_{2k}\right)}%
\global\long\def\Sn{S_{n}}%
\global\long\def\endsn{\mathrm{End}{}_{\Sn}}%
\global\long\def\xilambda{\sum_{\sigma\in S_{k}}\frac{d_{\lambda}}{k!}\chi^{\lambda}(\sigma)\left[v_{\sigma(1)}\otimes\dots\otimes v_{\sigma(k)}\right]}%
\global\long\def\p{\mathscr{P}}%
\global\long\def\dkn{\left\langle e_{i_{1}}\otimes\dots\otimes e_{i_{k}}:\ i_{1},\dots,i_{k}\ \mathrm{pairwise\ distinct}\right\rangle }%
\global\long\def\normxilambda{\xi_{\lambda}^{\mathrm{norm}}}%
\global\long\def\skdelta{S_{k}^{\psi}}%
\global\long\def\vlambdacheck{\check{V}^{\lambda}}%
\global\long\def\Q{\mathcal{Q}}%
\global\long\def\endoofvlambda{\check{\xi}_{\lambda}^{\mathrm{norm}}\otimes\normxilambda}%
\global\long\def\cntwok{\left(\C^{n}\right)^{\otimes2k}}%
\global\long\def\vsigmas{v_{\sigma(1)}\otimes\dots\otimes v_{\sigma(k)}\otimes v_{\sigma'(1)}\otimes\dots\otimes v_{\sigma'(k)}}%
\global\long\def\D{\mathcal{D}}%
\global\long\def\tablambda{\mathrm{Tab}_{\mathcal{D},\mathrm{sign}}\left(\lambda^{+}(2k)\right)}%
\global\long\def\part{\mathrm{Part}\left([2k]\right)}%
\global\long\def\ei{e_{i_{1}}\otimes\dots\otimes e_{i_{k}}}%
\global\long\def\csk{\C\left[S_{k}\right]}%
\global\long\def\mixedtensor{\cnk\otimes\left(\left(\C^{n}\right)^{\vee}\right)^{\otimes l}}%
\global\long\def\shortmixedtensor{C_{k,l}^{n}}%
\global\long\def\signisotypic{V_{\mathcal{D},\mathrm{sign}}^{\lambda^{+}(2k)}}%
\global\long\def\GLn{\mathrm{GL}_{n}}%

\title{Projection formulas and a refinement of Schur--Weyl--Jones duality
for symmetric groups}
\author{Ewan Cassidy}
\maketitle
\begin{abstract}
Schur--Weyl--Jones duality establishes the connection between the
commuting actions of the symmetric group $S_{n}$ and the partition
algebra $\P$ on the tensor space $\cnk.$ We use a refinement of
this considered first by Littlewood and later, by Sam and Snowden,
whereby there is a version of Schur--Weyl duality for the symmetric
groups $S_{n}$ and $S_{k}$ acting on a subspace of $\cnk$. We obtain
an explicit formula for the orthogonal projection from $\cnk$ to
each irreducible subrepresentation, yielding a new combinatorial approach
to computing stable irreducible characters of the symmetric group. 
\end{abstract}
\tableofcontents{}

\section{Introduction\label{sec:INTRODUCTION}}

Classic Schur--Weyl duality due to Schur \cite{Schur1,Schur2} asserts
that, in $\mathrm{End}\left(\cnk\right),$ $\GLn\left(\C\right)$
and $S_{k}$ generate full mutual centralizers of one another, when
$\GLn$ acts diagonally and $S_{k}$ permutes tensor coordinates.
Giving a decomposition into irreducible representations
\begin{equation}
\cnk\cong\bigoplus_{\lambda\vdash k,\ \ell(\lambda)\leq n}S^{\lambda}\otimes V^{\lambda},\label{eq: original schur weyl duality}
\end{equation}
where each $S^{\lambda}$ and $V^{\lambda}$ are irreducible representations
of $GL_{n}(\mathbb{C})$ and $S_{k}$ respectively that correspond
to the Young diagram $\lambda\vdash k$ (with $\ell(\lambda)$ non--empty
rows). This relates to a formula obtained earlier by Frobenius \cite{Frobenius},
giving the Schur polynomial expansion of the power sum symmetric polynomials.
Indeed, if $g\in\mathrm{GL}_{n}$ has eigenvalues $x_{1},\dots,x_{n}$
and $\sigma\in S_{k}$ has cycle type $\mu,$ then the trace of $g\sigma=\sigma g$,
viewed as an element of $\mathrm{End}\left(\cnk\right),$ is exactly
\begin{equation}
p_{\mu}\left(x_{1},\dots,x_{n}\right),\label{eq: bitrace general linear case}
\end{equation}
where $p_{\mu}$ is a power sum symmetric polynomial, 
\[
p_{\mu}(x_{1},\dots,x_{n})=\prod_{i=1}^{l(\mu)}\left(\sum_{j=1}^{n}x_{j}^{\mu_{i}}\right).
\]
Combining (\ref{eq: bitrace general linear case}) with (\ref{eq: original schur weyl duality})
yields the expansion 
\[
p_{\mu}=\sum_{\lambda\vdash k,\ l(\lambda)\leq n}\chi^{\lambda}\left(\mu\right)s_{\lambda},
\]
 where $\chi^{\lambda}$ is the character of $V^{\lambda}$ and $s_{\lambda}$
is the Schur polynomial corresponding to $\lambda.$ 

The analogous results in the case of $S_{n}$ acting diagonally on
$\cnk$ (i.e. considering only the diagonal action of \emph{permutation}
matrices in the above Schur--Weyl duality) were obtained by Jones
\cite{Jones} and can be expressed using the partition algebra, $\P.$
Denote by $\part$ the set of set partitions of $[2k]\overset{\mathrm{def}}{=}\{1,\dots,2k\}.$
The partition algebra $\P$ is the $\C$--linear span of $\part$
with a natural multiplication and a natural action on $\cnk,$ as
described in \S\ref{subsec:THE PARTITION ALGEBRA}. When $n\geq2k,$
$S_{n}$ and $\P$ generate full mutual centralizers of one another
in $\mathrm{End}\left(\cnk\right),$ leading to the decomposition
\[
\cnk\cong\bigoplus_{\lambda\vdash i,\ 0\leq i\leq k}V^{\lambda^{+}(n)}\otimes R^{\lambda},
\]
where $R^{\lambda}$ is the irreducible $\P$--representation corresponding
to the Young diagram $\lambda\vdash i$ and $V^{\lambda^{+}(n)}$
is the irreducible representation of $S_{n}$ indexed by the Young
diagram $\lambda^{+}(n),$ obtained from $\lambda$ by adding a row
of $n-|\lambda|$ boxes to the top, i.e. 
\[
\lambda^{+}(n)=(n-|\lambda|,\lambda)\vdash n,
\]
see \S\ref{subsec:SIMPLE MODULES FOR THE PARTITION ALGEBRA}. In
this paper, we use a refinement of this existing `Schur--Weyl--Jones
duality' due to Sam and Snowden \cite{SamSnowden} (considered also
by Littlewood \cite{Littlewood} in a slightly different context)
to construct a subspace
\[
\A\subseteq\cnk
\]
on which the inherited action of $\P$ descends to an action of the
embedded copy of 
\[
\csk\overset{\iota}{\hookrightarrow}\P,
\]
 viewing each permutation as a matching between $\{1,\dots,k\}$ and
$\{k+1,\dots,2k\}$. This inherited action of $\csk$ permutes tensor
coordinates. The subspace $\A$ also inherits an action of $S_{n}$
from $\cnk$ and we denote the associated representations by $\theta:S_{k}\to\mathrm{End}\left(\A\right)$
and $\rho:S_{n}\to\mathrm{End}\left(\A\right)$ respectively. These
actions commute, making $\A$ a $\C\left[S_{n}\times S_{k}\right]$--representation,
with decomposition
\begin{equation}
\A\cong\bigoplus_{\lambda\vdash k}V^{\lambda^{+}(n)}\otimes V^{\lambda}\label{eq: Ak(n) decomp introduction}
\end{equation}
 for $n\geq2k$. 

For each $\lambda\vdash k,$ we denote by $\mathcal{U}_{\lambda^{+}(n)}$
the subspace of $\A$ satisfying $\mathcal{U}_{\lambda^{+}(n)}\cong V^{\lambda^{+}(n)}\otimes V^{\lambda}.$
Our main result below is an explicit formula for the orthogonal projection
$\mathcal{Q}_{\lambda,n}$ from $\cnk$ to $\mathcal{U}_{\lambda^{+}(n)}$
for every $\lambda\vdash k$, written as a linear combination of $P_{\pi}^{\mathrm{strict}},$
an element of $\mathrm{End}\left(\cnk\right)$ defined by each $\pi\in\part$
as follows: 
\[
\left\langle P_{\pi}^{\mathrm{strict}}\left(e_{i_{1}}\otimes\dots\otimes e_{i_{k}}\right),e_{i_{k+1}}\otimes\dots\otimes e_{i_{2k}}\right\rangle =\begin{cases}
1 & \mathrm{if}\ j\sim k\ \mathrm{in}\ \pi\iff i_{j}=i_{k}\\
0 & \mathrm{otherwise.}
\end{cases}
\]
This is almost\footnote{There is a minor difference between the endomorphism defined by Jones,
denoted $\hat{\theta}(\pi)=P_{\pi}^{\mathrm{weak}}$ in this paper
(see \S\ref{subsec: DUALITY BETWEEN THE SYMMETRIC GROUP AND PARTITION ALGEBRA})
and $P_{\pi}^{\mathrm{strict}}.$ One can easily rewrite (\ref{eq: strict orthogonal projection formula})
in terms of the $P_{\pi}^{\mathrm{weak}},$ however the formula is
cleaner when written in terms of the $P_{\pi}^{\mathrm{strict}}$.
Additionally, writing it in this form makes it clear $\mathcal{Q}_{\lambda,n}(e_{i_{1}}\otimes\dots\otimes e_{i_{k}})=0$
unless $i_{1},\dots,i_{k}$ are all distinct. } equivalent to the action of $\P$ described by Jones \cite{Jones}
in the formulation of Schur--Weyl--Jones duality. As is described
in \S\ref{subsec:MOBIUS INVERSION}, there is a natural partial ordering
on $\part.$ We write $\pi_{1}\le\pi_{2}$ if $\pi_{1}$ can be obtained
by splitting the blocks of $\pi_{2}$ in to smaller blocks. Below,
we write $|\pi|$ for the number of blocks of the set partition $\pi$
and $d_{\lambda^{+}(n)}\overset{\mathrm{def}}{=}\dim V^{\lambda^{+}(n)}$
(this is $\sim n^{|\lambda|}$ for $n$ large enough). The set partitions
$\pi$ that appear in our formula are all of the form $\pi\le\iota(\tau)$
for some $\tau\in S_{k}.$ One should view these as a partial matching
between $\{1,\dots,k\}$ and $\{k+1,\dots,2k\},$ with some elements
left as singletons. These partitions span a distinct subalgebra of
the partition algebra -- in some sense they realize partial permutations
of $[k]$ inside $\P$. 
\begin{thm}
\label{thm: MAIN THEOREM}For any $k\in\mathbb{Z}_{>0},$ for any
$\lambda\vdash k$ and for any $n\geq2k,$ the orthogonal projection
\[
\Q_{\lambda,n}:\cnk\to\mathcal{U}_{\lambda^{+}(n)}
\]
is given by
\begin{equation}
d_{\lambda^{+}(n)}(-1)^{k}\sum_{\tau\in S_{k}}\chi^{\lambda}(\tau)\sum_{\pi\leq\iota(\tau)}\frac{(-1)^{|\pi|}}{(n)_{|\pi|}}P_{\pi}^{\mathrm{strict}}.\label{eq: strict orthogonal projection formula}
\end{equation}
\end{thm}

The operator $\mathcal{Q}_{\lambda,n}$ is a higher--order analogue
of the projector $P=\mathrm{Id}_{n}-\frac{1}{n}J$, with $J$ the
all $1$'s matrix, that projects from $\C^{n}$ to the $(n-1)$-dimensional
standard representation of $S_{n}$ by removing the trivial subrepresentation.
The standard representation is indexed by the Young diagram $(n-1,1),$
which is equivalent to $\lambda^{+}(n)$ when $\lambda=(1)\vdash1$,
and indeed, with $k=1,$ (\ref{eq: strict orthogonal projection formula})
gives exactly $P$. Further motivation for Theorem \ref{thm: MAIN THEOREM}
is detailed below.

\subsubsection*{Stable representations of the symmetric group}

For each fixed $k$ and $\lambda\vdash k,$ one obtains a family of
representations of $(S_{n})_{n\ge\lambda_{1}+k}$ indexed by the Young
diagrams $\{\lambda^{+}(n)\}_{n\ge\lambda_{1}+k}.$ Such families
are called stable irreducible representations and the collection of
all such characters is denoted $\widehat{S_{\infty}}.$ As is discussed
in \cite[Appendix B]{HananyPuder}, these characters form a basis
of the ring of class functions in $S_{n},$
\[
\mathbb{Q}\left[\eta_{1},\eta_{2},\dots\right],
\]
where 
\[
\eta_{i}(g)=\#\mathrm{fixed\ points\ of}\ g^{i}.
\]
 That is, any such character can be written as a polynomial in the
$\eta_{i}$ with rational coefficients and vice versa. For example,
the $(n-1)$-dimensional standard representation described above can
be written $\eta_{1}-1.$ These stable irreducible representations
are therefore closely related to the problem of counting cycles of
given fixed lengths in elements of $S_{n}.$ Moreover, as can be inferred
from e.g. \cite[Lemma 2.8]{Etingof}, for fixed $k$, the collection
of all Young diagrams $\lambda^{+}(n)$ with $|\lambda|\le k$ together
with their conjugates (obtained by transposing the rows/columns of
the Young diagram, or equivalently, tensoring with the one-dimensional
sign representation) index all irreducible representations of $S_{n}$
of polynomial dimension $\le Cn^{k}$. In practice, it is often useful
to be able to work explicitly with these `low-dimensional' irreducible
representations, for example in analyzing spectral properties of graphs
constructed using natural permutation actions of the symmetric group
(see e.g. \cite{HananyPuder,PArzanchevskiPuder}). 

The primary motivation for Theorem \ref{thm: MAIN THEOREM} is Corollary
\ref{cor: Corollary of main theorem} below. By viewing stable representations
as subrepresentations of tensor space $\cnk,$ the projection formula
provides an alternative approach to the use of $\eta_{1},\eta_{2},\dots,$
to compute the stable irreducible characters. Indeed, given $g\in S_{n}$,
one translates the problem of computing $\chi^{\lambda^{+}(n)}(g)$
to one of computing the trace in $\cnk$ of $g$ composed with $\mathcal{Q}_{\lambda,n}.$
The trace in $\cnk$ is easily understood (one can simply count how
many standard basis vectors are fixed) and the explicit description
of $\mathcal{Q}_{\lambda,n}$ as in (\ref{eq: strict orthogonal projection formula})
makes this approach viable. Below, we write $\mathrm{Tr}_{\mathcal{U}_{\lambda^{+}(n)}}$
for the trace of an endomorphism of $\cnk$ restricted to the subspace
$\mathcal{U}_{\lambda^{+}(n)}$ and $\mathrm{Tr}_{\cnk}$ for the
trace in the full tensor space. 
\begin{cor}
\label{cor: Corollary of main theorem}For any $g\in\C\left[S_{n}\right],\ \sigma\in\csk$, 

\begin{equation}
\mathrm{Tr}_{\mathcal{U}_{\lambda^{+}(n)}}\left(\rho(g)\theta(\sigma)\right)=\chi^{\lambda^{+}(n)}(g)\chi^{\lambda}(\sigma).\label{eq: bitrace equality on single block}
\end{equation}
In particular, taking $\sigma=e,$ we have 
\[
\begin{aligned}d_{\lambda}\chi^{\lambda^{+}(n)}(g) & =\mathrm{Tr}_{\mathcal{U}_{\lambda^{+}(n)}}\left(\rho(g)\theta(e)\right)\\
 & =\mathrm{Tr}_{\mathcal{U}_{\lambda^{+}(n)}}\left(\rho(g)\right)\\
 & =\mathrm{Tr}_{\cnk}\left(g\circ\mathcal{Q}_{\lambda,n}\right).
\end{aligned}
\]
\end{cor}

An intended application and some further motivation for this corollary
is given in \S\ref{SUBsubsec: AN APPLICATION}. We note that, in
\S\ref{subsec:DIMENSION OF Ak(n)}, we obtain a recursive formula
for the dimension of the space $\A$, namely, 
\[
\dim A_{k+1}(n)=n\dim A_{k}(n-1)-\dim A_{k}(n).
\]
Whilst this is certainly not our main result, it may be of independent
interest so we choose to highlight it here. 

\subsection{Overview of paper}

We give a rough outline of our construction of $\Q_{\lambda,n}$ and
some of the obstacles that arise. We first give the description of
$\A$ due to Sam and Snowden \cite{SamSnowden} in \S\ref{sec:SCHUR-WEYL DUALITY FOR Sn AND Sk}.
We define 
\[
\xi\overset{\mathrm{def}}{=}\element\in\A
\]
and then, for each $\lambda\vdash k,$ we define an element denoted
$\normxilambda\in\mathcal{U}_{\lambda^{+}(n)}.$ This is the (normalized)
projection of $\xi$ (via the $\csk$ action) from $\A$ to the $V^{\lambda}$-isotypic
component. This then tells us that the $S_{n}\times S_{k}$ representation
generated by $\normxilambda$ is exactly $\mathcal{U}_{\lambda^{+}(n)}$
and that the $S_{n}$ representation generated by $\normxilambda,$
\[
\mathcal{V}_{\lambda^{+}(n)}\overset{\mathrm{def}}{=}\left\langle \rho(g)\left(\normxilambda\right):\ g\in S_{n}\right\rangle ,
\]
is a subrepresentation of $\left(V^{\lambda^{+}(n)}\right)^{\oplus d_{\lambda}}.$ 

\subsubsection*{What did not work?}

As is the case in the analogous setting of $U(n)$ and $S_{k}\times S_{l}$
acting on the mixed tensor space $\cnk\otimes\left(\left(\C^{n}\right)^{\vee}\right)^{\otimes l}$
(see the construction given by Koike \cite{Koike1989} and further
detailed in \cite[Section 2.2]{Magee2021}) one may expect to be able
to show that $\mathcal{V}_{\lambda^{+}(n)}$ as defined above is irreducible.
Namely, that 
\[
\mathcal{V}_{\lambda^{+}(n)}\cong V^{\lambda^{+}(n)}.
\]
Were this to be the case, one could consider $\endoofvlambda\in\left(\mathcal{V}_{\lambda^{+}(n)}\right)^{\vee}\otimes\mathcal{V}_{\lambda^{+}(n)}\cong\mathrm{End}\left(\mathcal{V}_{\lambda^{+}(n)}\right).$
The element\footnote{Integrals in this paper are all done with respect to the uniform measure
on $S_{n}$. } 
\begin{equation}
\int_{S_{n}}\rho(g)\left(\normxilambda\right)^{\vee}\otimes\rho(g)\left(\normxilambda\right)dg\label{eq: projecting endo to Sn invars}
\end{equation}
is the projection of $\endoofvlambda$ to $\mathrm{End}_{S_{n}}\left(\mathcal{V}_{\lambda^{+}(n)}\right)$,
the space of endomorphisms of $\mathcal{V}_{\lambda^{+}(n)}$ that
commute with $\rho\left(\C\left[S_{n}\right]\right).$ Schur's lemma
then implies that (\ref{eq: projecting endo to Sn invars}) must be
some multiple of the identity map on $\mathcal{V}_{\lambda^{+}(n)}.$
Extending this by $0$ on the orthogonal complement of $\mathcal{V}_{\lambda^{+}(n)}$
in $\cnk$ then yields a scalar multiple of the orthogonal projection
from $\cnk\to\mathcal{V}_{\lambda^{+}(n)}.$ 

One could then compute (\ref{eq: projecting endo to Sn invars}) explicitly
using the \emph{Weingarten calculus for the symmetric group}, yielding
an explicit formula for this projection. The primary obstacle to this
approach is that, in general, $\mathcal{V}_{\lambda^{+}(n)}$ is not
irreducible, with multiplicity dependent on $\lambda$ and expressed
in terms of Kostka numbers and Littlewood--Richardson coefficients,
which are not easily computed. 

\subsubsection*{The work around}

To work around this difficulty, we use two key properties of $\xi$:
\begin{itemize}
\item For each of the transpositions $s_{1},s_{3},\dots,s_{2k-1}\in S_{2k}\le S_{n}$
(where $s_{i}=(i\ i+1)$), we have $\rho(s_{i})\left(\xi\right)=-\xi.$ 
\item For any $\lambda\vdash k$ and for any, $\sigma\in S_{k},$ $\theta\left(\sigma\right)\rho\left(\psi(\sigma)\right)\left(\normxilambda\right)=\normxilambda$
(i.e. $\psi(\sigma)\xi=\xi\sigma)$. Here, $\psi:S_{k}\to S_{2k}$
is as defined in \S\ref{subsec:HYPEROCTAHEDRAL GROUP}, we note that
this is \emph{not} the obvious inclusion obtained by adding $k$ fixed
points.
\end{itemize}
Our work around is to show that, when we restrict $\mathcal{U}_{\lambda^{+}(n)}$
to be a representation of the subgroup $\psi(S_{k})\times S_{k}\le S_{n}\times S_{k},$
then the multiplicity of the irreducible representation $V^{\lambda}\otimes V^{\lambda}$
is exactly $1$ and that $\normxilambda$ belongs to this subspace.
Below we describe how this is achieved through a series of restrictions. 

In \S\ref{SUBsubsec: RESTRICTING FROM Sn TO S2k} and \S\ref{SUBsubsec: PROPERTIES OF XI_=00005Clambda},
we use the first observation to show that when we restrict $\mathcal{U}_{\lambda^{+}(n)}$
to a representation of $S_{2k}\times S_{k},$ the irreducible block
$V^{\lambda^{+}(2k)}\otimes V^{\lambda}$ has multiplicity $1$ in
its decomposition in to irreducibles, and that $\normxilambda$ is
contained in this block. This then implies that the $S_{2k}\times S_{k}$
representation generated by $\normxilambda$, denoted $\mathcal{U}_{\lambda^{+}(2k)},$
is exactly isomorphic to $V^{\lambda^{+}(2k)}\otimes V^{\lambda}.$ 

The group $S_{2k}$ contains as a subgroup the \emph{hyperoctahedral
group }$H_{k},$ see \S\ref{subsec:HYPEROCTAHEDRAL GROUP}. The irreducible
representations of $H_{k}$ are indexed by pairs of Young diagrams
$(\mu,\nu)$ of total size $k$ and are denoted $W^{\mu,\nu}.$ It
is not too hard to show that, when we restrict $\mathcal{U}_{\lambda^{+}(2k)}$
to the subgroup $H_{k}\times S_{k}$ and decompose into irreducibles,
then $\normxilambda$ belongs to the $W^{\emptyset,\lambda}$-isotypic
subspace. One of the more technical challenges of this paper is to
prove the following proposition. 
\begin{prop}
\label{prop: restrict V^l=00005Cmabda^+(2k) to H_k}In the restriction
of $V^{\lambda^{+}(2k)}$ to the subgroup $H_{k}\le S_{2k},$ the
multiplicity of $W^{\emptyset,\lambda}$ in its decomposition in to
irreducibles is exactly $1.$
\end{prop}

Proposition \ref{prop: restrict V^l=00005Cmabda^+(2k) to H_k} is
proved in \S\ref{subsec:CONSTRUCTING W^=00005Cphi,=00005Clambda =00005Cotimes V^=00005Clambda INSIDE CNK},
using the first observation again and the Vershik--Okounkov approach
to the representation theory of the symmetric group (i.e. via Young
tableaux), see \S\ref{subsec:SYMMETRIC GROUP}. We highlight that
Proposition \ref{prop: restrict V^l=00005Cmabda^+(2k) to H_k} can,
in principle, be derived from the general branching rule between $S_{2k}$
and $H_{k}$ as described by Koike and Terada \cite{KoikeTerada1987}
in terms of plethysm of symmetric polynomials. However, working with
such objects is notoriously difficult, and it does not appear that
our specific multiplicity--one statement is explicitly recorded in
the existing literature. 

With Proposition \ref{prop: restrict V^l=00005Cmabda^+(2k) to H_k}
in hand, it is straightforward to show that the $\psi\left(S_{k}\right)\times S_{k}$
representation generated by $\normxilambda$ is isomorphic to $V^{\lambda}\otimes V^{\lambda}$.
From this point onward, we use the second observation to show that
(\ref{eq: projecting endo to Sn invars}) is a multiple of the orthogonal
projection $\mathcal{Q}_{\lambda,n}:\cnk\to\mathcal{U}_{\lambda^{+}(n)}$.
This is done in \S\ref{subsec:IDENTIFYING THE PROJECTION}. To complete
the proof of Theorem \ref{thm: MAIN THEOREM}, it remains to evaluate
(\ref{eq: projecting endo to Sn invars}) explicitly using the Weingarten
calculus and we do this in \S\ref{subsec:CALCULATING THE PROJECTION}. 

\subsection{Intended applications and further questions\label{subsec:INTENDED APPLICATIONS AND FURTHER QUESTIONS}}

\subsubsection{What if $n<2k?$\label{SUBsubsec:WHAT IF n<2k?}}

An assumption for much of this paper is that $n\geq2k.$ In this case,
$\P$ injects in to $\mathrm{End}\left(\cnk\right)$, but for $n<2k$,
this is not true. Indeed, in \cite[Theorem 2.6]{ComesOstrik}, the
kernel of the map $\P\to\mathrm{End}\left(\cnk\right)$ is determined
to be spanned by the elements 
\[
\{x_{\pi}\in\P:\ |\pi|>n\}
\]
where 
\[
x_{\pi}\overset{\mathrm{def}}{=}\pi-\sum_{\mu>\pi}x_{\mu}
\]
and $\leq$ is the natural partial ordering on set partitions (as
in \S\ref{subsec:THE PARTITION ALGEBRA})\footnote{We thank Brendon Rhoades for directing us to this result.}.
Note that this agrees with the injectivity of $\P\to\mathrm{End}_{S_{n}}\left(\cnk\right)$
in the case of $n\geq2k.$ Obtaining a similar expression for the
kernel of the map $\csk\to\mathrm{End}_{S_{n}}\left(\A\right)$ could
yield a decomposition of $\A$ into irreducible blocks in the case
$n<2k$, which would be worth pursuing in future for completeness
(for $n\geq2k,$ this map is injective and can be used to decompose
$\A$ as in (\ref{eq: Ak(n) decomp introduction})). 

\subsubsection{An application and further motivation\label{SUBsubsec: AN APPLICATION}}

An intended application for this paper is in the computation of the
expected characters\emph{ }of $w$--random permutations. This question
has been addressed in e.g. \cite{Puder2011,HananyPuder,PuderParzanchewski2012},
as well as \cite{MageePuder1,MageePuder2} where the same question
is considered for other compact groups. Given a word $w\in F_{m}=\left\langle x_{1},\dots,x_{m}\right\rangle $
and a compact group $G$, one obtains a word map 
\[
w:\underbrace{G\times\dots\times G}_{m}\to G,
\]
 whereby $w\left(g_{1},\dots,g_{m}\right)$ is the element of $G$
obtained by replacing each $x_{i}$ in $w$ with $g_{i}$. Our projection
formula is intended to be applied in the case of $G=S_{n}.$

Determining the distribution on $S_{n}$ induced by each word map
$w$ is, in most cases, highly non--trivial and a natural starting
point is to consider, for each $w\in F_{m},$ 

\[
\mathop{\mathbb{E}}_{\sigma_{1},\dots,\sigma_{m}\in S_{n}}\left[\#\mathrm{fixed\ points\ of\ }\left(w\left(\sigma_{1},\dots,\sigma_{m}\right)\right)\right],
\]
where $\sigma_{1},\dots,\sigma_{m}\in S_{n}$ are i.i.d. uniformly
random. Puder and Parzanchevski \cite{PuderParzanchewski2012} give
sharp asymptotic bounds for the expected number of fixed points in
terms of the \emph{primitivity rank }of $w$, an algebraic invariant
of $w$ introduced by Puder in \cite{Puder2011}. A word $w\in F_{m}$
is said to be primitive in a free group if it belongs to some basis
of that group and the primitivity rank $\pi(w)$ of a word in $F_{m}$
is defined by
\[
\pi(w)\overset{\mathrm{def}}{=}\mathrm{min}\Big\{\mathrm{rk}H:\ H\leq F_{m},\ w\in H,\ w\ \mathrm{not\ primitive\ in}\ H\Big\}.
\]
If no such subgroup exists, then we set $\pi(w)=\infty.$ 
\begin{thm}[{\cite[Theorem 1.8]{PuderParzanchewski2012}}]
 For any $w\in F_{m},$ 
\[
\mathop{\mathbb{E}}_{\sigma_{1},\dots,\sigma_{m}\in S_{n}}\left[\#\mathrm{fixed\ points\ of\ }\left(w\left(\sigma_{1},\dots,\sigma_{m}\right)\right)\right]=1+\frac{|\mathrm{Crit}(w)|}{n^{\pi(w)-1}}+O\left(\frac{1}{n^{\pi(w)}}\right).
\]
 
\end{thm}

Hanany and Puder generalised this to all stable irreducible representations
of $S_{n}.$
\begin{thm}[{\cite[Theorem 1.3]{HananyPuder}}]
 For any $k\in\mathbb{Z}_{\geq2}$, for any $\lambda\vdash k$ and
for any $w\in F_{m}$ that is not a proper power or the identity,
\[
\mathop{\mathbb{E}}_{\sigma_{1},\dots,\sigma_{m}\in S_{n}}\left[\chi^{\lambda^{+}(n)}\left(w\left(\sigma_{1},\dots,\sigma_{m}\right)\right)\right]=O\left(\frac{1}{n^{\pi(w)}}\right).
\]
\end{thm}

Moreover, they conjecture the following much stronger bound.
\begin{conjecture}
For any $k\in\mathbb{Z}_{\geq2}$, for any $\lambda\vdash k$ and
for any $w\in F_{m}$
\[
\mathop{\mathbb{E}}_{\sigma_{1},\dots,\sigma_{m}\in S_{n}}\left[\chi^{\lambda^{+}(n)}\left(w\left(\sigma_{1},\dots,\sigma_{m}\right)\right)\right]=O\left(\frac{1}{\left(\dim V^{\lambda^{+}(n)}\right)^{\pi(w)-1}}\right).
\]
\end{conjecture}

Corollary \ref{cor: Corollary of main theorem} gives a combinatorial
method for computing the irreducible characters $\chi^{\lambda^{+}(n)}$
and the explicit description of the coefficients of the $P_{\pi}^{\mathrm{strict}}$
in (\ref{eq: strict orthogonal projection formula}) should prove
useful for improving the asymptotic bounds for 
\[
\mathop{\mathbb{E}}_{\sigma_{1},\dots,\sigma_{m}\in S_{n}}\left[\chi^{\lambda^{+}(n)}\left(w\left(\sigma_{1},\dots,\sigma_{m}\right)\right)\right],
\]
 using analytic means.
\begin{rem}
In subsequent work \cite{Cassidy2}, we have used Theorem \ref{thm: MAIN THEOREM}
to prove that, for any word $w\in F_{m}$ that is not a proper power,
\[
\mathop{\mathbb{E}}_{\sigma_{1},\dots,\sigma_{m}\in S_{n}}\left[\chi^{\lambda^{+}(n)}\left(w\left(\sigma_{1},\dots,\sigma_{m}\right)\right)\right]=O\left(\frac{1}{\left(\dim V^{\lambda^{+}(n)}\right)}\right),
\]
solving one aspect of the conjecture (and confirming the conjecture
for $w\in F_{m}$ with $\pi(w)=2$, thus also resolving the conjecture
for $F_{2}$), whilst the full conjecture remains open. 
\end{rem}

\subsection*{Acknowledgements}

We thank Michael Magee for his guidance and comments on this paper.
This paper is part of a project that received funding from the European
Research Council (ERC) under the European Union\textquoteright s Horizon
2020 research and innovation programme (grant agreement No 949143).
We thank the anonymous referees for their careful reading of the manuscript
and for their helpful comments and suggestions.

\section{Background}

\subsection{Preliminaries and notation}

\subsubsection{Representation Theory\label{subsec:REPRESENTATION THEORY}}

For a thorough introduction, we direct the reader to the textbooks
by Fulton and Harris \cite{FultonHarris} or Serre \cite{Serre},
for example. Given a representation $(\rho,V)$ of a finite group
$G,$ we will denote by $\left(\rho^{*},V^{\vee}\right)$ its \emph{dual
representation}, where $V^{\vee}$ is the space of linear functionals
on $V$ and $\rho^{*}$ is defined by $\rho^{*}(g)\left(\check{v}\right)\left[\rho(g)(u)\right]=\check{v}(u)$.
Given a decomposition $V\cong\bigoplus V_{i}^{\oplus a_{i}}$ of an
arbitrary representation $V$ of a finite group $G,$ we define 
\begin{equation}
\p_{V_{i}}\overset{\mathrm{def}}{=}\frac{\mathrm{dim}(V_{i})}{|G|}\sum_{g\in G}\overline{\chi_{V_{i}}(g)}g\in\C[G],\label{arbitraryprojectioneqn}
\end{equation}
a central idempotent in the group algebra whose image $\rho\left(\p_{V_{i}}\right)\in\mathrm{End}(V)$
is the projection from $V$ on to the $V_{i}$-isotypic subspace.

We denote $[n]\overset{\mathrm{def}}{=}\{1,\dots,n\}$ and throughout
this paper, $e_{1},\dots,e_{n}\in\C^{n}$ will be the standard orthonormal
(with respect to the standard Hermitian inner product $\langle.,.\rangle$)
basis of $\C^{n}$ so that the set 
\[
\big\{ e_{i_{1}}\otimes\dots\otimes e_{i_{k}}:i_{j}\in[n]\ \mathrm{for}\ j\in[k]\big\}
\]
is the standard basis for $\cnk.$ Given a multi--index $I=\left(i_{1},\dots,i_{k}\right),$
we will write 
\[
e_{I}\overset{\mathrm{def}}{=}e_{i_{1}}\otimes\dots\otimes e_{i_{k}}\in\left(\C^{n}\right)^{\otimes k}
\]
and similarly $\check{e}_{I}$ denotes $\check{e}_{i_{1}}\otimes\dots\otimes\check{e}_{i_{k}}\in\left(\left(\C^{n}\right)^{\vee}\right)^{\otimes k}.$
Given a representation $\rho:G\to\mathrm{End}(V)$ and a subgroup
$H\leq G,$ we will write $\mathrm{Res}_{H}^{G}V$ or $V\downarrow_{H}$
for the restriction of $V$ to $H$ and, given a representation $U$
of $H,$ we will denote by $\mathrm{Ind}_{H}^{G}U$ the induced representation
of $G$. 

\subsubsection{Symmetric Group\label{subsec:SYMMETRIC GROUP}}

We denote the symmetric group on $n$ elements by $S_{n}=\{\mathrm{bijections}\ [n]\to[n]\}$
together with function composition. Where we refer to an inclusion
of the form $S_{m}\subseteq S_{n},$ for $m\leq n,$ unless specified
otherwise, we refer to the subgroup consisting of permutations in
$S_{n}$ which fix the elements $\{m+1,\dots,n\}$. In this case,
we may also refer to an inclusion $S'_{n-m}\subseteq S_{n},$ which
refers to the subgroup consisting of permutations that fix the elements
$\{1,\dots,m\}.$ It can be shown (see \cite{Humphreys} for example)
that $S_{n}$ is generated by transpositions of adjacent elements,
\emph{
\[
S_{n}\cong\left\langle (12),(23),\dots,(n-1\ n)\right\rangle .
\]
}These are the \emph{Coxeter generators, }denoted $s_{i}=(i\ i+1)$
for $i=1,\dots,n-1$. 

A \emph{Young Diagram }(YD) $\lambda$ is an arrangement of rows of
boxes, where the number of boxes in each row is non--increasing as
the row index increases. If $\lambda$ is a YD with $n$ boxes and
$l(\lambda)$ non--empty rows of boxes, then we write $\lambda\vdash n$
(or $|\lambda|=n)$ and we say that the length of $\lambda$ is $l(\lambda)$.
We can write this as $\lambda=\left(\lambda_{1},\dots,\lambda_{l(\lambda)}\right),$
with $\lambda_{1}\geq\dots\geq\lambda_{l(\lambda)}>0$ and $\lambda_{1}+\dots+\lambda_{l(\lambda)}=n.$
There is a notion of \emph{inclusion} for Young diagrams $\lambda$
and $\mu,$ where $|\mu|\leq|\lambda|.$ Informally, we say $\mu$
is contained inside $\lambda$ if we can obtain $\lambda$ by adding
boxes to $\mu$ (equivalently, removing boxes from $\lambda$ to obtain
$\mu$). More formally, if $\mu=\left(\mu_{1},\dots,\mu_{l(\mu)}\right)$
and $\lambda=\left(\lambda_{1},\dots,\lambda_{l(\lambda)}\right)$,
then $\mu$ is contained in $\lambda$ if $l(\mu)\leq l(\lambda)$
and for each $i\in[l(\mu)],$ we have $\mu_{i}\leq\lambda_{i}.$ Where
$\mu$ is contained in $\lambda,$ we define the \emph{skew diagram
}$\lambda\backslash\mu$ to be the diagram consisting of the boxes
that are in $\lambda$, but not in $\mu.$ There is a natural bijection
between distinct isomorphism classes of irreducible representations
of $S_{n}$ and Young diagrams $\lambda\vdash n$, for example see
\cite[Chapter 4]{FultonHarris} or \cite{VershikOkounkov}. As such,
we label every irreducible representation of $S_{n}$ by its corresponding
YD $\lambda$, and denote this representation by $V^{\lambda}$. We
will denote the character of $V^{\lambda}$ by $\chi^{\lambda}$.
\begin{rem}
The characters of the symmetric group are integer valued (see \cite[p. 103]{Serre}
for example). Consequently, all representations $V$ of the symmetric
group are \emph{self dual, }meaning that $V\cong V^{\vee}$, since
\[
\chi_{V^{\vee}}(g)=\overline{\chi_{V}(g)}=\chi_{V}(g).
\]
\end{rem}

We will denote by $d_{\lambda}\overset{\mathrm{def}}{=}\chi^{\lambda}\left(\mathrm{Id}\right)=\dim V^{\lambda}.$
A \emph{standard tableau }or \emph{Young tableau of shape $\mathcal{\lambda}\vdash n$
}is a labeling of the boxes of $\lambda$ with the integers $1,\dots,n,$
in which every integer appears exactly once and the numbers are increasing
both along the rows and down the columns. The set of standard tableau
of shape $\lambda$ is denoted by $\mathrm{Tab}(\lambda)$ and it
is a fact that $|\mathrm{Tab}(\lambda)|=d_{\lambda}.$ The content,
$\mathrm{cont}\left(\Box\right)$, of a box $\Box$ in a YD $\lambda$
is defined by
\[
\mathrm{cont}\left(\Box\right)=\mathrm{column\ index\ of\ \Box\ -\ row\ index\ of\ \Box.}
\]
Given $T\in\mathrm{Tab}(\lambda),$ we define a \emph{content vector
$(c_{1},\dots,c_{n})\in\mathbb{Z}^{n}$, }where $c_{i}$ is the content
of the box labeled $i$. Then we have the following proposition of
Vershik and Okounkov \cite[Proposition 6.2]{VershikOkounkov}.
\begin{prop}
\label{prop:vershik okounkov basis of V^=00005Clambda}There exists
an orthonormal basis $\{v_{T}\}_{T\in\mathrm{Tab}(\lambda)}$ of $V^{\lambda}$
in which the Coxeter generators act according to the following rules:
\begin{itemize}
\item If the boxes labeled $i$ and $i+1$ are in the same row of $T$,
then $s_{i}v_{T}=v_{T};$ 
\item If the boxes labeled $i$ and $i+1$ are in the same column of $T$,
then $s_{i}v_{T}=-v_{T}$;
\item If the boxes labeled $i$ and $i+1$ are in neither the same row or
the same column of $T$, then $s_{i}$ acts on the two dimensional
space spanned by $v_{T}$ and $v_{T^{\prime}}$ (the Young tableau
obtained from $T$ by swapping $i$ and $i+1$) by the following matrix
\[
\left(\begin{array}{cc}
r^{-1} & \sqrt{1-r^{-2}}\\
\sqrt{1-r^{-2}} & -r^{-1}
\end{array}\right),
\]
where $r=c_{i+1}-c_{i}$.
\end{itemize}
\end{prop}

From now on, for each $\lambda$ we will fix such a basis of $V^{\lambda},$
and refer to this as the \emph{Gelfand--Tsetlin }basis. Additionally,
given some irreducible representation $V^{\lambda}\otimes V^{\mu}$
of $S_{|\lambda|}\times S_{|\mu|},$ the basis $\{v_{T_{1}}\otimes v_{T_{2}}\}_{T_{1}\in\mathrm{Tab}(\lambda),\ T_{2}\in\mathrm{Tab}(\mu)}$
formed by taking all tensor products of Gelfand--Tsetlin basis vectors
of $V^{\lambda}$ and $V^{\mu}$ will be referred to as the Gelfand--Tsetlin
basis of $V^{\lambda}\otimes V^{\mu}.$

\subsubsection{Hyperoctahedral Group\label{subsec:HYPEROCTAHEDRAL GROUP}}

We give an overview of the hyperoctahedral group, $H_{k},$ viewed
as a subgroup of $S_{2k}$, and its representation theory. This brief
introduction comprises of a summary similar to that of Koike and Terada
in \cite{KoikeTerada1987}. A very thorough introduction to this topic
can be found in \cite{Musili}. There is an injective group homomorphism
\begin{equation}
\psi:S_{k}\to S_{2k}\label{eq: Sk delta first time}
\end{equation}
whereby each Coxeter generator $s_{i}=(i\ i+1)$ is mapped to $\psi(s_{i})=\big(2i-1\ 2i+1\big)\big(2i\ 2i+2\big).$
We will write 
\[
\skdelta=\psi\left(S_{k}\right)\leq S_{2k}
\]
and the hyperoctahedral group is then defined to be the following
subgroup of $S_{2k}$:
\begin{equation}
H_{k}=\Big\langle\psi\left(S_{k}\right),\ s_{1},\ s_{3},\dots,\ s_{2k-1}\Big\rangle\cong\psi\left(S_{k}\right)\ltimes\mathcal{D},\label{eq: hyperoctahedral defin}
\end{equation}
where 
\[
\mathcal{D}=\big\langle s_{1},\ s_{3},\dots,\ s_{2k-1}\big\rangle\leq S_{2k}.
\]
For each $i=0,\ 1,\ 3,\dots,\ 2k-1,$ define a representation $\rho_{i}$
of $\mathcal{D}$ by 
\begin{equation}
\chi_{\rho_{i}}\left(s_{j}\right)=\begin{cases}
1 & \mathrm{if}\ j\leq i\\
-1 & i<j.
\end{cases}\label{eq: description of rho_i for Hk}
\end{equation}
Given an irreducible representation $V^{\lambda}$ of $S_{k},$ define
an irreducible representation $W^{\lambda,\emptyset}$ of $H_{k}$
via (\ref{eq: hyperoctahedral defin}) as follows -- given $h\in H_{k}$,
we have $h=\psi(\sigma)\delta$ with $\sigma\in S_{k}$ and $\delta\in\mathcal{D}$,
and the character $\chi^{\lambda,\emptyset}$ of $W^{\lambda,\emptyset}$
is given by $\chi^{\lambda,\emptyset}(h)=\chi^{\lambda,\emptyset}(\psi(\sigma)\delta)=\chi^{\lambda}(\sigma).$
We can also define an irreducible representation of $H_{k}$ from
$\rho_{0},$ by extending $\rho_{0}$ by letting $\psi\left(S_{k}\right)$
act trivially. We denote this representation of $H_{k}$ by $W^{\emptyset,(k)}$
and we define 
\begin{equation}
W^{\emptyset,\lambda}=W^{\lambda,\emptyset}\otimes W^{\emptyset,(k)},\label{eq: definition of OUR irreducible Hk representation}
\end{equation}
with character given by $\chi^{\emptyset,\lambda}=\mathrm{sign}(\delta)\chi^{\lambda}(\sigma).$\footnote{This construction only works with $\rho_{0}$ since it is $\psi\left(S_{k}\right)$
invariant, but does not work for more general $\rho_{i}.$} 

Given any $i=1,\ 3,\dots,\ 2k-1,$ define $j_{i}\overset{\mathrm{def}}{=}\frac{i+1}{2}$
and define two subgroups of $H_{k},$ 
\[
H_{j_{i}}=\Big\langle\psi\left(S_{j_{i}}\right),\ s_{1},s_{3},\dots,\ s_{i}\Big\rangle
\]
and 
\[
H_{k-j_{i}}^{\prime}=\Big\langle\psi\left(S_{k-j_{i}}^{\prime}\right),s_{i+2},\ s_{i+4},\dots,s_{2k-1}\Big\rangle,
\]
where $S_{k-\frac{i+1}{2}}^{\prime}$ is defined as in \S\ref{subsec:SYMMETRIC GROUP}.
These are obviously isomorphic to smaller index hyperoctahedral groups,
and we get the following theorem (see e.g. \cite[Theorem 4.7.7]{Musili})
using the `Wigner--Mackey method of little groups'. 
\begin{thm}
\label{thm:construction of irreps for Hk}With the subgroups and representations
defined as above,
\begin{enumerate}
\item[a)]  for any $i=1,\ 3,\dots,\ 2k-1$ and Young diagrams $\mu\vdash j_{i}$,
$\nu\vdash k-j_{i},$ the representation 
\[
W^{\mu,\nu}=\mathrm{Ind}_{H_{j_{i}}\times H'_{k-j_{i}}}^{H_{k}}\Big[W^{\mu,\emptyset}\times\left(W^{\nu,\emptyset}\otimes W^{\emptyset,(k-j_{i})}\right)\Big]
\]
is an irreducible representation of $H_{k}$;
\item[b)]  the set 
\[
\big\{ W^{\mu,\nu}:\ (\mu,\nu)\models k\big\}
\]
 constitutes a complete set of representatives of the equivalence
classes of irreducible representations of $H_{k},$ where $(\mu,\nu)\models k$
denotes an ordered pair of Young diagrams $\mu,\nu$ with $|\mu|+|\nu|=k.$
\end{enumerate}
\end{thm}

\begin{rem}
In the case of $\mu=\emptyset,$ it is easily seen that $\mathrm{dim}\left(W^{\emptyset,\nu}\right)=d_{\nu},$
the dimension of the irreducible $S_{k}$ representation, $V^{\nu}.$ 
\end{rem}

\subsubsection{M{\"o}bius Inversion\label{subsec:MOBIUS INVERSION}}

Here we give definitions and results that are needed for our description
of the Weingarten calculus in \S\ref{SUBsubsec:THE WEINGARTEN CALCULUS FOR THE SYMMETRIC GROUP}.
Most of the details can be found in the foundational paper of Rota
\cite{Rota}. A \emph{poset }$(P,\leq)$ is a set $P$ with a partial
order $\leq$\emph{. }In general, we will simply write $P$ in place
of $(P,\leq)$ in reference to a poset, unless it is necessary to
be explicit. A \emph{lattice }is a poset in which the maximum and
minimum of two elements is defined and they will be called the \emph{join
}and \emph{meet }respectively. We denote the join of two elements
$x,y$ by $x\vee y$ and we denote the meet of these two elements
by $x\wedge y$. Let $P$ be a poset and $x,y\in P.$ A segment $[x,y]$
is defined as follows: 
\[
[x,y]\overset{\mathrm{def}}{=}\{z\in P:\ x\leq z\leq y\}.
\]
Open and half open segments are defined similarly. We say that $P$
is \emph{locally finite }if every segment $[x,y]$ contains finitely
many elements. The M{\"o}bius function $\mu(x,y)$ of a poset $P$
is defined inductively in \cite[Proposition 1]{Rota}. For a segment
$[x,y]$ of a poset $P$, we first set $\mu(x,x)=1$. Then, assuming
$\mu(x,z)$ is defined for all $z\in[x,y),$ we inductively define
\begin{equation}
\mu(x,y)=-\sum_{x\leq z<y}\mu(x,z).\label{eq:mobius function definition}
\end{equation}
The M{\"o}bius inversion formula below is given in \cite[Corollary 1]{Rota}. 
\begin{thm}[M{\"o}bius Inversion Formula]
\label{thm:mobius inversion formula}Let $(P,\leq)$ be a locally
finite poset and let $r:P\to\C.$ Suppose there exists a $p\in P$
such that $x>p\implies r(x)=0.$ Then, if 
\[
s(x)=\sum_{y\geq x}r(y)
\]
 we have 
\[
r(x)=\sum_{y\geq x}\mu(x,y)s(y).
\]
\end{thm}

\subsubsection{The Partition Algebra\label{subsec:THE PARTITION ALGEBRA}}

Jones \cite{Jones}, and independently Martin \cite{Martin1994},
initially developed the partition algebra in relation to statistical
mechanics and it has since been used to develop various formulations
of Schur--Weyl duality. We write 
\[
\part=\left(\{\mathrm{set\ partitions\ of}\ [2k]\},\ \leq\right),
\]
where $\pi_{1}\leq\pi_{2}$ if $\pi_{1}$ is a refinement of $\pi_{2},$
meaning every block of $\pi_{1}$ is contained in a block of $\pi_{2}.$
In this setting, $\part$ is a finite (and hence locally finite) lattice.
Given $\pi\in\part$, we write $i\sim j$ to indicate that $i$ and
$j$ belong to the same block of the partition. 

\begin{rem}
Given this partial ordering, one can find the M�bius function for
partitions using the inductive definition (\ref{eq:mobius function definition}).
Suppose $\pi_{1}=\Big\{\mathcal{S}_{1},\dots,\mathcal{S}_{l}\Big\}$
consists of $l$ subsets and that $\pi_{2}$ is a refinement of $\pi_{1}$,
with each subset $\mathcal{S}_{i}$ of $\pi_{1}$ splitting into a
further $m_{i}$ subsets, $\mathcal{T}_{1},\dots,\mathcal{T}_{m_{i}}.$
So $\pi_{2}$ consists of $\sum_{i}m_{i}=m$ subsets. Then 
\[
\mu(\pi_{1},\pi_{2})=(-1)^{m-l}\prod_{i=1}^{l}(m_{i}-1)!,
\]
see \cite[Section 7]{Rota}. 
\end{rem}

The number $1\leq s\leq2k$ of blocks of $\pi\in\part)$ is the \emph{size}
of the partition, denoted $|\pi|$. Each set partition $\pi\in\part$
corresponds to a partition diagram consisting of two rows of vertex
labeled $1,\dots,k$ on the top row and $k+1,\dots,2k$ on the bottom
row, with edges between vertices in the same block $\pi$. Obviously,
any such diagram also defines some $\pi\in\part$. We describe a multiplication
of $\pi_{1},\pi_{2}\in\part$ as follows, let $x$ be an indeterminate.
\footnote{The diagram for $\pi_{1}\wedge\pi_{2}$ is obtained by removing any
edges in the diagram of $\pi_{1}$ that are not present in that of
$\pi_{2}.$ The diagram for the join $\pi_{1}\vee\pi_{2}$ is obtained
by adding all edges of $\pi_{2}$ to $\pi_{1},$ and then completing
every connected component.}
\begin{enumerate}
\item Identify the bottom row of vertices in $\pi_{1}$ with the top row
of vertices in $\pi_{2}$ to obtain a diagram with $3$ rows of $k$
vertices. 
\item Let $\gamma$ be the number of connected components of this diagram
with vertices in only the middle row. 
\item Add edges between any two vertices in the same connected component,
if there is not already an edge. 
\item Remove the middle row of vertices and any adjacent edges to vertices
in the middle row to obtain a new partition diagram. Label this diagram
$\pi_{3}$. 
\item Define $\pi_{1}\pi_{2}=x^{\gamma}\pi_{3}$. 
\end{enumerate}
\begin{example*}
For $k=3,$ $\pi_{1}=\Big\{\{1,2\},\{3,5,6\},\{4\}\Big\}$ and $\pi_{2}=\Big\{\{1\},\{2,4\},\{3,6\},\{5\}\Big\}$,
then $\pi_{1}\pi_{2}=x\Big\{\{1,2\},\{3,4,6\},\{5\}\Big\}.$ Diagrammatically, 

\begin{figure}[H]
\centering
\begin{tikzpicture}[main/.style = draw, circle] 
\node[main, fill, inner sep=1pt] (1) {};
\node[main, fill, inner sep=1pt] (2) [right of =1] {};
\node[main, fill, inner sep=1pt] (3) [right of =2] {};
\node(4) [below right of =3] {$\times$};
\node[main, fill, inner sep=1pt] (5) [above right of=4] {};
\node[main, fill, inner sep=1pt] (6) [right of=5]{};
\node[main, fill, inner sep=1pt] (7) [right of=6] {};
\node(8) [below right of =7] {$=$};
\node(21) [right of =8] {$x$};
\node[main, fill, inner sep=1pt] (9) [above right of=21] {};
\node[main, fill, inner sep=1pt] (10) [right of=9] {};
\node[main, fill, inner sep=1pt] (11) [right of=10] {};

\node[main, fill, inner sep=1pt] (12) [below left of =4]{};
\node[main, fill, inner sep=1pt] (13) [left of =12] {};
\node[main, fill, inner sep=1pt] (14) [left of =13] {};
\node[main, fill, inner sep=1pt] (15) [below right of =4] {};
\node[main, fill, inner sep=1pt] (16) [right of =15] {};
\node[main, fill, inner sep=1pt] (17) [right of =16] {};
\node[main, fill, inner sep=1pt] (18) [below right of =21]{};
\node[main, fill, inner sep=1pt] (19) [right of =18] {};
\node[main, fill, inner sep=1pt] (20) [right of =19] {};

\draw (1)--(2) (3)--(12) (3)--(13) (13)--(12) (6)--(15) (7)--(17) (9)--(10) (11)--(18) (11)--(20) (18) to [out=335,in=225,looseness=1] (20);
\end{tikzpicture}
\end{figure}
\end{example*}
If $\mathbb{C}(x)$ is the field of rational functions with complex
coefficients, the partition algebra $P_{k}(x)$ is the $\mathbb{C}(x)$--linear
span of $\part$. With the multiplication as described above, this
is an associative algebra, with identity element: 
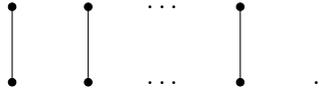
\begin{figure}[H]
\centering
\begin{tikzpicture}[main/.style = draw, circle] 
\node[main, fill, inner sep=1pt] (1) {};
\node[main, fill, inner sep=1pt] (2) [right of =1] {};
\node (3) [right of =2]{$\dots$};
\node[main, fill, inner sep=1pt] (4)[right of =3] {};

\node[main, fill, inner sep=1pt] (5) [below of=1]{};
\node[main, fill, inner sep=1pt] (6) [right of =5] {};
\node (7) [right of =6]{$\dots$};
\node[main, fill, inner sep=1pt] (8)[right of =7] {};
\draw (1)--(5) (2)--(6) (4)--(8);.
\end{tikzpicture}
\end{figure}
 For each $n\in\mathbb{C}$, we define the partition algebra $P_{k}(n)$
over $\mathbb{C}$ as the linear span of $\part$, with $x$ replaced
by $n$ in the multiplication described. For most choices of $n,$
this is a semisimple algebra. 
\begin{thm}[\cite{MartinSaleur1992}]
The partition algebra $P_{k}(n)$ is semisimple for any $n\in\C,$
unless $n\in\mathbb{Z}\cap[0,2k-1].$
\end{thm}

\begin{prop}[{\cite[Proposition 1]{Martin1996}}]
The following elements generate $\P:$
\begin{itemize}
\item $\mathrm{Id}=\Big\{\{1,k+1\},\dots,\{k,2k\}\Big\}$
\item $\mathcal{S}_{i,j}=\Big\{\{1,k+1\},\dots\{i,k+j\},\dots,\{j,k+i\},\dots,\{k,2k\}\Big\}$
for $i,j=1,\dots,k$ with $i\neq j$
\item $\mathcal{A}_{i}=\Big\{\{1,k+1\},\dots,\{i\},\{k+i\},\dots,\{k,2k\}\Big\}$
for $i=1,\dots,k$ and
\item $\mathcal{A}_{i,j}=\Big\{\{1,k+1\},\dots,\{i,j,k+i,k+j\},\dots,\{k,2k\}\Big\}$
for $i,j=1,\dots,k$ with $i\neq j$.
\end{itemize}
\end{prop}

There is a surjection 
\begin{equation}
R:\P\to\C\left[S_{k}\right],\label{eq: restricting Pkn to CS_k}
\end{equation}
whereby $\mathcal{S}_{i,i+1}\mapsto s_{i}$ and $\mathcal{A}_{i},\mathcal{A}_{i,j}\mapsto0$
and a corresponding algebra injection 
\begin{equation}
\iota:\C\left[S_{k}\right]\to\P,\label{eq: Including CS_k  in Pkn}
\end{equation}
where, for $\sigma\in S_{k},$

\[
\iota(\sigma)=\Big\{\{1,k+\sigma^{-1}(1)\},\dots,\{k,k+\sigma^{-1}(k)\}\Big\}.
\]
It is corresponding in the sense that $R\circ\iota$ is the identity
map on $\C\left[S_{k}\right]$. Intuitively, each permutation in $S_{k}$
corresponds to some matching of the two rows of vertices in the diagram.
Henceforth, any reference to the inclusion or restriction between
$\C\left[S_{k}\right]$ and $\P$ will reference maps (\ref{eq: restricting Pkn to CS_k})
and (\ref{eq: Including CS_k  in Pkn}). 

\subsubsection{The Weingarten Calculus for the Symmetric Group\label{SUBsubsec:THE WEINGARTEN CALCULUS FOR THE SYMMETRIC GROUP}}

Here, we describe the method for integrating over $S_{n}$ outlined
by Collins, Matsumoto and Novak in the short survey \cite{CollinsWeingartenShort},
in the language to be used in this paper. The goal is to explicitly
compute 
\[
w\overset{\mathrm{def}}{=}\int_{S_{n}}e_{g(i_{1})}\otimes\dots\otimes e_{g(i_{k})}dg\in\cnk,
\]
with respect to the uniform measure. We write 
\[
w=\sum_{1\leq j_{1},\dots,j_{k}\leq n}\left(\alpha_{j_{1},\dots,j_{k}}\right)e_{j_{1}}\otimes\dots\otimes e_{j_{k}}
\]
and define, for each $\pi\in\mathrm{Part}([k]),$ a linear functional
$\pi^{\mathrm{strict}}:\cnk\to\C$ where, for each $I=(i_{1},\dots,i_{k}),$
\begin{equation}
\pi^{\mathrm{strict}}\left(e_{I}\right)=\begin{cases}
1 & \mathrm{if\ }j\sim l\ \mathrm{in}\ \pi\iff e_{i_{j}}=e_{i_{l}}\\
0 & \mathrm{otherwise.}
\end{cases}\label{eq:pi strict definition}
\end{equation}
In general, there exists a $g\in S_{n}$ such that 
\begin{equation}
e_{g(i_{1})}\otimes\dots\otimes e_{g(i_{k})}=e_{j_{1}}\otimes\dots\otimes e_{j_{k}},\label{eq: e_g(I) =00003D e_J}
\end{equation}
\emph{if and only if} the multi--index $I=\left(i_{1},\dots,i_{k}\right)$
and the multi--index $J=\left(j_{1},\dots,j_{k}\right)$ `define
the same partition' i.e. there is exactly one partition $\pi\in\mathrm{Part}\left([k]\right)$
for which $\pi^{\mathrm{strict}}\left(e_{I}\right)\pi^{\mathrm{strict}}\left(e_{J}\right)=1.$
In this case, there are exactly $(n-|\pi|)!$ permutations $g\in S_{n}$
satisfying (\ref{eq: e_g(I) =00003D e_J}), leading to the following
proposition. 
\begin{prop}
\label{thm:weingartenstrict} The coefficient $\alpha_{j_{1},\dots,j_{k}}$
of $e_{j_{1}}\otimes\dots\otimes e_{j_{k}}$ in $w$ is 
\begin{equation}
\sum_{\pi\in\mathrm{Part}\left([k]\right)}\pi^{\mathrm{strict}}\left(e_{i_{1}}\otimes\dots\otimes e_{i_{k}}\right)\pi^{\mathrm{strict}}\left(e_{j_{1}}\otimes\dots\otimes e_{j_{k}}\right)\frac{1}{(n)_{|\pi|}},\label{eq:weingartenstrict}
\end{equation}
where $(n)_{|\pi|}=n(n-1)\dots(n-|\pi|+1)$ is the Pochhammer symbol.
\end{prop}

This can be alternatively formulated using M{\"o}bius inversion,
as in \cite[Theorem 1.3 and Proposition 1.4]{BanicaCurran2010}. Define
the linear functional $\pi^{\mathrm{weak}}:\cnk\to\C$ , where, for
each $I=(i_{1},\dots,i_{k}),$ 
\begin{equation}
\pi^{\mathrm{weak}}\left(e_{I}\right)=\begin{cases}
1 & \mathrm{if}\ j\sim l\ \mathrm{in}\ \pi\implies e_{i_{j}}=e_{i_{l}}\\
0 & \mathrm{otherwise.}
\end{cases}\label{eq:pi weak definition}
\end{equation}
Obviously, 
\[
\pi^{\mathrm{strict}}\left(e_{i_{1}}\otimes\dots\otimes e_{i_{k}}\right)\neq0\implies\pi^{\mathrm{weak}}\left(e_{i_{1}}\otimes\dots\otimes e_{i_{k}}\right)\neq0
\]
and in fact, 
\begin{equation}
\pi^{\mathrm{weak}}=\sum_{\pi_{1}\geq\pi}\pi_{1}^{\mathrm{strict}}.\label{eq:pi weak in terms of strict}
\end{equation}
Using Theorem \ref{thm:mobius inversion formula}, we obtain the formula
\begin{equation}
\pi^{\mathrm{strict}}=\sum_{\pi_{1}\geq\pi}\mu(\pi,\pi_{1})\pi_{1}^{\mathrm{weak}}\label{eq:pi strict in terms of weak}
\end{equation}
and can thus rewrite (\ref{eq:weingartenstrict}) as 
\[
\sum_{\pi_{1},\pi_{2}\in\mathrm{Part}\left([k]\right)}\pi_{1}^{\mathrm{weak}}\left(e_{I}\right)\pi_{2}^{\mathrm{weak}}\left(e_{J}\right)\mathrm{Wg_{n,k}}\left(\pi_{1},\pi_{2}\right),
\]
where $\mathrm{Wg}_{n,k}$ is the Weingarten function for $S_{n}$,
given by 
\[
\mathrm{Wg}_{n,k}(\pi_{1},\pi_{2})=\sum_{\pi\leq\pi_{1}\wedge\pi_{2}}\mu(\pi,\pi_{1})\mu(\pi,\pi_{2})\frac{1}{(n)_{|\pi|}}.
\]

\subsection{Schur--Weyl--Jones duality\label{subsec: DUALITY BETWEEN THE SYMMETRIC GROUP AND PARTITION ALGEBRA}}

\subsubsection{Duality between $\protect\P$ and $S_{n}$\label{subsec:DUALITY BETWEEN Pk(n) AND Sn}}

Given $\pi\in\part,$ define a right action of $\pi$ on $\cnk$ by
\begin{equation}
\left\langle \left(e_{i_{1}}\otimes\dots\otimes e_{i_{k}}\right)\pi,\ e_{i_{k+1}}\otimes\dots\otimes e_{i_{2k}}\right\rangle =\begin{cases}
1 & \mathrm{if}\ j\sim k\ \mathrm{in}\ \pi\implies i_{j}=i_{k}\\
0 & \mathrm{otherwise.}
\end{cases}\label{eq:action of pkn on cnk}
\end{equation}
Extending linearly gives $\cnk$ a right $\P$-module structure and
defines, for each $\pi\in\P$, an element 
\[
P_{\pi}^{\mathrm{weak}}\in\mathrm{End}\left(\cnk\right),
\]
where $P_{\pi}^{\mathrm{weak}}(v)=v\pi.$ We will denote the map $\pi\mapsto P_{\pi}^{\mathrm{weak}}$
by 
\[
\hat{\theta}:\P\to\mathrm{End}\left(\cnk\right).
\]
 
\begin{rem}
\label{rem: P_pi strict definition}We have 
\[
P_{\pi}^{\mathrm{weak}}=\sum_{J,I}\pi^{\mathrm{weak}}\left(e_{J}\otimes e_{I}\right)\left(\check{e}_{J}\otimes e_{I}\right)
\]
via the canonical isomorphism $\mathrm{End}\left(\cnk\right)\cong\left(\check{\C}^{n}\right)^{\otimes k}\otimes\cnk,$
where the sum is over all multi--indices $J,I$ of size $k.$ We
also define $P_{\pi}^{\mathrm{strict}}\in\mathrm{End}\left(\cnk\right)$
by 
\[
\left\langle P_{\pi}^{\mathrm{strict}}\left(e_{i_{1}}\otimes\dots\otimes e_{i_{k}}\right),\ e_{i_{k+1}}\otimes\dots\otimes e_{i_{2k}}\right\rangle =\begin{cases}
1 & \mathrm{if}\ j\sim k\ \mathrm{in}\ \pi\iff i_{j}=i_{k}\\
0 & \mathrm{otherwise,}
\end{cases}
\]
noting that 
\[
P_{\pi}^{\mathrm{strict}}=\sum_{J,I}\pi^{\mathrm{strict}}\left(e_{J}\otimes e_{I}\right)\left(\check{e}_{J}\otimes e_{I}\right).
\]
This is the form used in the statement of our main theorem. 
\end{rem}

Now, $\cnk$ is also a left $S_{n}$-module, as the $k^{\mathrm{th}}$
tensor power of the defining representation of $S_{n}.$ Denote this
representation 
\[
\hat{\rho}:\C\left[S_{n}\right]\to\mathrm{End}\left(\cnk\right),
\]
so that, for any $\tau\in S_{n}$ and basis vector $\ei\in\cnk,$
we have 
\[
\hat{\rho}(\tau)\left(\ei\right)=e_{\tau(i_{1})}\otimes\dots\otimes e_{\tau(i_{k})}.
\]
Schur--Weyl--Jones duality asserts that these actions generate full
mutual centralizers of one another in $\mathrm{End}\left(\cnk\right).$
\begin{thm}[\cite{Jones}]
\label{partitionduality}For $n\geq2k,$ where $\P$ acts via $\hat{\theta}$
and $S_{n}$ acts via $\hat{\rho},$ 
\begin{enumerate}
\item $\P$ generates $\mathrm{End}_{S_{n}}\left(\cnk\right)$
\item $S_{n}$ generates $\mathrm{End}_{\P}\left(\cnk\right)$.
\end{enumerate}
\end{thm}

\subsubsection{Simple Modules for the Partition Algebra\label{subsec:SIMPLE MODULES FOR THE PARTITION ALGEBRA}}

For each $\lambda\vdash k,$ recall the notation $\lambda^{+}(n)\vdash n$
for the Young diagram given by $(n-k,\lambda)$, defining a family
of irreducible representations of $S_{n}$ for each $n\geq\lambda_{1}.$
Given any YD $\mu,$ we denote by $\mu^{*}$ the YD obtained by removing
the first row of boxes (with this notation, $\left(\lambda^{+}(n)\right)^{*}=\lambda$).
Martin and Saleur \cite[Corollary 4.1]{MartinSaleur1992} showed that
the simple modules of $\P$ are parametrized by Young diagrams of
size $\leq k.$
\begin{thm}
When $\P$ is semisimple, with 
\[
\{\lambda\vdash i:i=0,\dots,k\}=\Lambda_{k,n},
\]
the set
\[
\big\{ R^{\lambda}:\ \lambda\in\Lambda_{k,n}\big\}
\]
constitutes a full set of representatives of the isomorphism classes
of simple $\P$-modules. 
\end{thm}

Using double centralizer theory, the following decomposition follows
from Theorem \ref{partitionduality}.
\begin{thm}
\label{partitiondecomposition}For each $k\in\mathbb{Z}_{>0}$, for
every $n\geq2k,$ as a $\left(S_{n},\P\right)$-bimodule, 
\[
\cnk\cong\bigoplus_{\lambda\in\Lambda_{k,n}}V^{\lambda^{+}(n)}\otimes R^{\lambda}.
\]
\end{thm}

\subsection{Schur--Weyl Duality for $S_{n}$ and $S_{k}$\label{sec:SCHUR-WEYL DUALITY FOR Sn AND Sk}}

We now present the refinement of Schur--Weyl--Jones duality due
to Sam and Snowden \cite[Section 6.1.3.]{SamSnowden} which is important
for our construction of the projection $\mathcal{Q}_{\lambda,n}.$
This construction was also considered by Littlewood \cite{Littlewood}
in a somewhat different language to the one used in this paper. 

\subsubsection{Constructing $\protect\A$}

For each $1\leq j\leq k$, define the $j^{\mathrm{th}}$ linear contraction
map
\[
T_{j}:\cnk\to\left(\mathbb{C}^{n}\right)^{\otimes k-1},
\]
where 
\[
T_{j}\left(e_{I}\right)=e_{i_{1}}\otimes\dots\otimes\dot{e}_{i_{j}}\otimes\dots\otimes e_{i_{k}},
\]
using the notation 
\[
e_{i_{1}}\otimes\dots\otimes\dot{e}_{i_{j}}\otimes\dots\otimes e_{i_{k}}\overset{\mathrm{def}}{=}e_{i_{1}}\otimes\dots\otimes e_{i_{j-1}}\otimes e_{i_{j+1}}\otimes\dots\otimes e_{i_{k}}.
\]
We also define 
\[
D_{k}(n)\overset{\mathrm{def}}{=}\dkn\subseteq\cnk.
\]

\begin{defn}
\label{def:definition of Ak(n)}We define a vector subspace 
\[
\A\overset{\mathrm{def}}{=}D_{k}(n)\cap\bigcap_{j=1}^{k}\mathrm{Ker}\left(T_{j}\right)\subseteq\cnk.
\]
\end{defn}

\noindent This space is clearly invariant under the inherited action
of $S_{n}$ and we denote this representation 
\[
\rho:\C\left[S_{n}\right]\to\mathrm{End}\left(\A\right).
\]
Definition \ref{def:definition of Ak(n)} can be reformulated as follows.
Define a function 
\[
\mathrm{pn}:\part\to\mathbb{Z}_{\geq0}
\]
where $\mathrm{pn}(\pi)$ is the \emph{propagating number }(see \cite[Definition 5]{Martin1996})
of $\pi$ -- it is the number of subsets of $\pi$ that contain at
least one element $i\leq k$ \emph{and }at least one element $j$
with $k+1\leq j\leq2k$. Such subsets correspond to connected components
in the diagram of $\pi$ with at least one vertex from each row. Two
obvious properties are:
\begin{itemize}
\item For any $\pi\in\part,$ $0\leq\mathrm{pn}(\pi)\leq k$;
\item If $\mathrm{pn}(\pi)=k$ then $\pi=\iota(\tau)$ for some $\tau\in S_{k}.$
\end{itemize}
Let 
\begin{equation}
I_{k}(n)\overset{\mathrm{def}}{=}\left\langle \pi:\pi\in\part,\ \mathrm{pn}(\pi)\leq k-1\right\rangle _{\C}\subseteq\P\label{eq: I_k(n) definition}
\end{equation}
be the ideal generated by all $\pi\in\part$ with propagating number
$\leq k-1$. So, $I_{k}(n)$ is the kernel of the map $R:\P\to\csk$
(recall (\ref{eq: restricting Pkn to CS_k})) which yields the isomorphism
\[
\csk\cong\P/I_{k}(n).
\]
Then we also have 
\[
\A=\bigcap_{\pi\in I_{k}(n)}\ker\left(P_{\pi}^{\mathrm{weak}}\right),
\]
the subspace of $\cnk$ annihilated by $I_{k}(n).$ The inherited
action of $\P$ on $\A$ thus descends to an action of $\csk$ and
this action permutes tensor coordinates. We denote the associated
representation 
\[
\theta:\csk\to\mathrm{End}\left(\A\right),
\]
i.e. $\theta(\sigma)(w_{1}\otimes\dots\otimes w_{k})=\hat{\theta}\left(\iota\left(\sigma^{-1}\right)\right)(w_{1}\otimes\dots\otimes w_{k})=w_{\sigma^{-1}(1)}\otimes\dots\otimes w_{\sigma^{-1}(k)}$.
Defining the representation 
\[
\Delta:\C\left[S_{n}\times S_{k}\right]\to\mathrm{End}\left(\A\right)
\]
by $\Delta(g,\sigma)\overset{\mathrm{def}}{=}\rho(g)\theta(\sigma)=\theta(\sigma)\rho(g),$
we get a decomposition of $\A$ in to irreducible subrepresentations
\begin{equation}
\A\cong\bigoplus_{\lambda\vdash k}V^{\lambda^{+}(n)}\otimes V^{\lambda}.\label{eq: A_k(n) decomp}
\end{equation}

\section{Projections\label{sec:PROJECTIONS}}

\subsection{Overview of Section \ref{sec:PROJECTIONS}\label{subsec:OVERVIEW OF SECTION 3}}

Recall the definition 
\[
\xi\overset{\mathrm{def}}{=}\element=v_{1}\otimes\dots\otimes v_{k},
\]
where we write $v_{i}\overset{\mathrm{def}}{=}e_{i}-e_{i+1}$. To
discuss the aims of this section we begin with the following lemma,
which details how $\xi$ is not properly contained in some subset
of the irreducible blocks that make up $\A.$ 
\begin{lem}
\label{lem: XI has non-zero projection to each isotypic subspace}For
any $\lambda\vdash k$ and for any $n\geq2k$, the projection $\xi_{\lambda}$
of $\xi$ to the $V^{\lambda}$-isotypic component of $\A$ is $\neq0$.
\end{lem}

\begin{proof}
The projection of $\xi$ to the $V^{\lambda}$-isotypic component
of $A_{k}(n)$ is 
\[
\begin{aligned}\xi_{\lambda} & \overset{\mathrm{def}}{=}\theta\left(\p_{V^{\lambda}}\right)\left(\xi\right)\\
 & =\frac{d_{\lambda}}{k!}\sum_{\sigma\in S_{k}}\chi^{\lambda}(\sigma)\left[v_{\sigma(1)}\otimes\dots\otimes v_{\sigma(k)}\right].
\end{aligned}
\]
Observe that, for $\sigma\neq\tau,$ we have 
\[
\left\langle \theta(\sigma)(\xi),\ \theta(\tau)(\xi)\right\rangle =\prod_{i=1}^{k}\left\langle v_{\sigma(i)},\ v_{\tau(i)}\right\rangle =0,
\]
since at least one of these factors must be zero. So, $\{\theta(\sigma)\left(\xi\right):\sigma\in S_{k}\}$
is a linearly independent set of vectors and, as a sum of these vectors
with non--zero coefficients, $\xi_{\lambda}$ is non--zero itself. 
\end{proof}
For our purposes, it will be useful to normalize $\xi_{\lambda}$.
The norm $||\xi_{\lambda}||=\langle\xi_{\lambda},\xi_{\lambda}\rangle^{\frac{1}{2}}$
is easily computed:

\begin{align*}
\langle\xi_{\lambda},\xi_{\lambda}\rangle & =\left\langle \xilambda,\sum_{\tau\in S_{k}}\frac{d_{\lambda}}{k!}\chi^{\lambda}(\tau)\left[v_{\tau(1)}\otimes\dots\otimes v_{\tau(k)}\right]\right\rangle \\
 & =\left(\frac{d_{\lambda}}{k!}\right)^{2}\sum_{\sigma\in S_{k}}\sum_{\tau\in S_{k}}\chi^{\lambda}(\sigma)\chi^{\lambda}(\tau)\left\langle v_{\sigma(1)}\otimes\dots\otimes v_{\sigma(k)},v_{\tau(1)}\otimes\dots\otimes v_{\tau(k)}\right\rangle \\
 & =\left(\frac{d_{\lambda}}{k!}\right)^{2}\sum_{\sigma\in S_{k}}\chi^{\lambda}(\sigma)^{2}\left\langle v_{\sigma(1)},v_{\sigma(1)}\right\rangle \dots\left\langle v_{\sigma(k)},v_{\sigma(k)}\right\rangle \\
 & =\left(\frac{d_{\lambda}}{k!}\right)^{2}2^{k}\sum_{\sigma\in S_{k}}\chi^{\lambda}(\sigma)^{2}\\
 & =\frac{2^{k}d_{\lambda}^{2}}{k!}.
\end{align*}
We will write 
\begin{eqnarray*}
\normxilambda & \overset{\mathrm{def}}{=}\frac{\xi_{\lambda}}{||\xi_{\lambda}||} & =\left(\frac{k!}{2^{k}d_{\lambda}^{2}}\right)^{\frac{1}{2}}\sum_{\sigma\in S_{k}}\frac{d_{\lambda}}{k!}\chi^{\lambda}(\sigma)\left[v_{\sigma(1)}\otimes\dots\otimes v_{\sigma(k)}\right]\\
 &  & =\left(\frac{1}{2^{k}k!}\right)^{\frac{1}{2}}\sum_{\sigma\in S_{k}}\chi^{\lambda}(\sigma)\left[v_{\sigma(1)}\otimes\dots\otimes v_{\sigma(k)}\right].
\end{eqnarray*}
It is then immediate from Lemma \ref{lem: XI has non-zero projection to each isotypic subspace}
that 
\[
\mathcal{U}_{\lambda^{+}(n)}\overset{\mathrm{def}}{=}\Big\langle\Delta\left(g,\sigma\right)\left(\normxilambda\right):\ g\in S_{n},\ \sigma\in S_{k}\Big\rangle_{\C}\cong V^{\lambda^{+}(n)}\otimes V^{\lambda}.
\]
Now recall that $S_{k}^{\psi}\le S_{2k}\le S_{n}$ is defined as an
embedded copy 
\[
S_{k}\overset{\psi}{\hookrightarrow}S_{2k}\le S_{n}
\]
via the map $\psi$ as defined in \S\ref{subsec:HYPEROCTAHEDRAL GROUP}.
Our goal in Sections \ref{SUBsubsec: RESTRICTING FROM Sn TO S2k}--\ref{subsec:CONSTRUCTING W^=00005Cphi,=00005Clambda =00005Cotimes V^=00005Clambda INSIDE CNK}
is to prove that the $S_{k}^{\psi}\times S_{k}$ representation generated
by $\normxilambda$, 
\[
\left\langle \Delta(g,\sigma)\left(\normxilambda\right):\ g\in S_{k}^{\psi},\ \sigma\in S_{k}\right\rangle ,
\]
 is irreducible and isomorphic to $V^{\lambda}\otimes V^{\lambda}.$
This is done by considering a series of restrictions and keeping track
of the isotypic components containing $\normxilambda$ and their multiplicities: 
\begin{itemize}
\item In \S\ref{SUBsubsec: RESTRICTING FROM Sn TO S2k}, we restrict $V^{\lambda^{+}(n)}$
to $S_{2k}$ and determine that the $V^{\lambda^{+}(2k)}$-isotypic
component has multiplicity one.
\item In \S\ref{SUBsubsec: PROPERTIES OF XI_=00005Clambda}, we determine
that $\normxilambda$ belongs to the irreducible subspace $V^{\lambda^{+}(2k)}\otimes V^{\lambda}$
of $\mathcal{U}_{\lambda^{+}(n)}$ and hence that the $S_{2k}\times S_{k}$
representation generated by $\normxilambda$ is exactly isomorphic
to $V^{\lambda^{+}(2k)}\otimes V^{\lambda}$ (this is Proposition
\ref{prop: V^=00005Clambda^+(2k) =00005Cotimes V^=00005Clambda generated by xi}). 
\item In \S\ref{subsec:CONSTRUCTING W^=00005Cphi,=00005Clambda =00005Cotimes V^=00005Clambda INSIDE CNK},
we use the Gelfand--Tsetlin basis of standard tableaux on the Young
diagram $\lambda^{+}(2k)$ to determine that the multiplicity of the
irreducible $H_{k}$ representation, $W^{\emptyset,\lambda}$, has
multiplicity one in the restriction of $V^{\lambda^{+}(2k)}$ to the
subgroup $H_{k}\le S_{2k}.$ It is then straightforward to establish
that the $S_{k}^{\psi}\times S_{k}$ representation generated by $\normxilambda$
is indeed irreducible and isomorphic to $V^{\lambda}\otimes V^{\lambda}.$
\end{itemize}
In \S\ref{subsec:IDENTIFYING THE PROJECTION} we use this fact to
prove that 
\[
\int_{S_{n}}\rho^{*}(g)\left(\check{\xi}_{\lambda}^{\mathrm{norm}}\right)\otimes\rho(g)\left(\normxilambda\right)dg,
\]
viewed as an element of $\mathrm{End}\left(\cnk\right)$ via the canonical
isomorphism $\left(\check{\C}^{n}\right)^{\otimes k}\otimes\cnk\cong\mathrm{End}\left(\cnk\right)$,
is a scalar multiple of the orthogonal projection $\mathcal{Q}_{\lambda,n}:\cnk\to\mathcal{U}_{\lambda^{+}(n)}.$
The remainder of the paper uses the Weingarten calculus for the symmetric
group to compute this explicitly, writing $\mathcal{Q}_{\lambda,n}$
in the form 
\[
\sum_{\pi\in\part}c(n,k,\lambda,\pi)P_{\pi}^{\mathrm{strict}},
\]
as in Theorem \ref{thm: MAIN THEOREM}.

\subsubsection{Dimension of $A_{k}(n)$\label{subsec:DIMENSION OF Ak(n)}}

We digress briefly to find a recursive formula for the dimension of
$A_{k}(n)$, when $n\geq2k+1.$ Since 
\[
\A=\bigoplus_{\lambda\vdash k}\mathcal{U}_{\lambda^{+}(n)},
\]
then, alongside Lemma \ref{lem: XI has non-zero projection to each isotypic subspace},
it is clear that 
\[
\left\langle \Delta(g,\sigma)\left(\xi\right):\ g\in S_{n},\ \sigma\in S_{k}\right\rangle _{\C}=\A.
\]
Now, consider the subspace $B_{k+1}(n)$ of $\left(\C^{n}\right)^{\otimes k+1}$,
\[
B_{k+1}(n)\overset{\mathrm{def}}{=}D_{k+1}(n)\cap\bigcap_{j=1}^{k}\ker(T_{j})\subseteq\left(\C^{n}\right)^{\otimes k+1}.
\]
Restricting 
\[
T_{k+1}:\left(\C^{n}\right)^{\otimes k+1}\to\cnk
\]
to 
\[
\hat{T}_{k+1}:B_{k+1}\to\cnk,
\]
it is obvious that $\mathrm{Im}\left(\hat{T}_{k+1}\right)\subseteq\A.$
It is also clear that $\A\subseteq\mathrm{Im}\left(\hat{T}_{k+1}\right)$,
since 
\[
\hat{T}_{k+1}\left(\element\otimes e_{2k+1}\right)=\xi.
\]
The inclusion $A_{k+1}(n)\hookrightarrow B_{k+1}(n)$ gives an exact
sequence
\[
0\to A_{k+1}(n)\hookrightarrow B_{k+1}(n)\overset{\hat{T}_{k+1}}{\to}A_{k}(n)\to0,
\]
from which we obtain a recursive formula for the dimension of $A_{k+1}(n)$:
\[
\dim\left(A_{k+1}(n)\right)=\dim\left(B_{k+1}(n)\right)-\dim\left(A_{k}(n)\right).
\]
Now, $B_{k+1}(n)\cong\bigoplus_{i=1}^{n}A_{k}(n-1)\otimes e_{i}$.
To see this, consider the vector space 
\[
A_{k}^{i}(n)=\left\langle e_{i_{1}}\otimes\dots\otimes e_{i_{k}}:\ i_{j}\in\{1,\dots,i-1,i+1,\dots,n\},\ i_{j}\ \mathrm{all\ distict}\right\rangle \cap\bigcap_{j=1}^{k}\ker\left(T_{j}\right).
\]
Then, for each $i=1,\dots,n,$
\[
A_{k}^{i}(n)\cong A_{k}(n-1)
\]
and there is an obvious isomorphism 
\[
\bigoplus_{i=1}^{n}A_{k}^{i}(n)\otimes e_{i}\cong B_{k+1}(n).
\]
From this observation we see that $\mathrm{dim}\left(B_{k+1}(n)\right)=n\mathrm{dim}\left(A_{k}(n-1)\right),$
which yields the formula 
\begin{equation}
\dim A_{k+1}(n)=n\dim A_{k}(n-1)-\dim A_{k}(n).\label{akndimension}
\end{equation}
 The dimensions of $\A$ for $k=0,\dots,10$ are in the table below,
expressed as a polynomial in $n$. 
\begin{center}
\begin{table}[H]
\centering
\noindent\caption{The dimension of $\protect\A$ as a polynomial in $n$ for fixed $k$}
\begin{tabular}{|c|c|}
\hline 
{\footnotesize$k$} & {\footnotesize$\mathrm{dim}\left(\A\right)$}\tabularnewline
\hline 
\hline 
{\footnotesize$1$} & {\footnotesize$n-1$}\tabularnewline
\hline 
{\footnotesize$2$} & {\footnotesize$n^{2}-3n+1$}\tabularnewline
\hline 
{\footnotesize$3$} & {\footnotesize$n^{3}-6n^{2}+8n-1$}\tabularnewline
\hline 
{\footnotesize$4$} & {\footnotesize$n^{4}-9n^{3}+22n^{2}-13n+1$}\tabularnewline
\hline 
{\footnotesize$5$} & {\footnotesize$n^{5}-12n^{4}+43n^{3}-49n^{2}+18n-1$}\tabularnewline
\hline 
{\footnotesize$6$} & {\footnotesize$n^{6}-15n^{5}+71n^{4}-122n^{3}+87n^{2}-23n+1$}\tabularnewline
\hline 
{\footnotesize$7$} & {\footnotesize$n^{7}-18n^{6}+106n^{5}-245n^{4}+265n^{3}-136n^{2}+28n-1$}\tabularnewline
\hline 
{\footnotesize$8$} & {\footnotesize$n^{8}-21n^{7}+148n^{6}-431n^{5}+630n^{4}-491n^{3}+196n^{2}-33n+1$}\tabularnewline
\hline 
{\footnotesize$9$} & {\footnotesize$n^{9}-24n^{8}+197n^{7}-693n^{6}+1281n^{5}-1351n^{4}+819n^{3}-267n^{2}+38n-1$}\tabularnewline
\hline 
{\footnotesize$10$} & {\footnotesize$n^{10}-27n^{9}+253n^{8}-1044n^{7}+2338n^{6}-3122n^{5}+2562n^{4}-1268n^{3}+349n^{2}-43n+1$}\tabularnewline
\hline 
\end{tabular}
\end{table}
\par\end{center}

\subsection{Identifying a Projection Map\label{subsec:IDENTIFYING A PROJECTION MAP}}

\subsubsection{Restricting from $S_{n}$ to $S_{2k}$ \label{SUBsubsec: RESTRICTING FROM Sn TO S2k}}

We write the decomposition of $V^{\lambda^{+}(n)}\downarrow_{S_{2k}}$
into irreducible representations of $S_{2k}$ as 
\begin{equation}
V^{\lambda^{+}(n)}\downarrow_{S_{2k}}\cong\bigoplus_{\mu\vdash2k}\left(V^{\mu}\right)^{\oplus c_{\mu}}.\label{eq: restriction Sn to S2k  GENERAL FORM}
\end{equation}

\emph{Pieri's rule} is a `branching rule' for irreducible representations
of the symmetric group, see \cite[p. 59]{FultonHarris} for example.
It asserts that the multiplicity of $V^{\mu}$ in (\ref{eq: restriction Sn to S2k  GENERAL FORM})
is zero, unless the YD $\mu$ is contained inside the YD $\lambda^{+}(n)$.
In this case, the multiplicity is the number of ways of labeling the
skew YD $\lambda^{+}(n)\backslash\mu$ with the numbers $1,\dots,n-2k,$
so that no number is repeated and the numbers are increasing both
along each row and down each column. Using this rule, if $\lambda\vdash k$
is any fixed Young diagram with $k$ boxes: 
\begin{itemize}
\item if $\mu\vdash2k$ with $\mu_{1}<k$, then $V^{\mu}$ has multiplicity
$0$ in (\ref{eq: restriction Sn to S2k  GENERAL FORM}), since then
$|\mu^{*}|>k$, meaning $\mu^{*}$ is not contained within $\lambda$; 
\item if $\mu\vdash2k$ has $\mu_{1}=k,$ then $V^{\mu}$ has multiplicity
zero, unless $\mu=\lambda^{+}(2k)$ \emph{and} 
\item the multiplicity of $V^{\lambda^{+}(2k)}$ is $1$ -- the skew diagram
$\lambda^{+}(n)\backslash\lambda^{+}(2k)$ is exactly one row of $n-2k$
boxes. 
\end{itemize}
These observations imply the following proposition.
\begin{prop}
\label{prop:isomorphism of S2k x Sk module generated by xi - general partitions }For
any $\lambda\vdash k$ and $n\geq2k$, we have 
\[
\mathcal{U}_{\lambda^{+}(n)}\downarrow_{S_{2k}\times S_{k}}=\mathcal{U}_{\lambda^{+}(2k)}\oplus\mathcal{U}_{\lambda^{+}(2k)}^{\perp},
\]
where 
\begin{equation}
\mathcal{U}_{\lambda^{+}(2k)}\cong V^{\lambda^{+}(2k)}\otimes V^{\lambda}\label{eq: isomorphism of S2k x Sk module general}
\end{equation}
and 
\[
\mathcal{U}_{\lambda^{+}(2k)}^{\perp}\cong\bigoplus_{\mu\vdash2k,\ \mu_{1}>k}\left(V^{\mu}\otimes V^{\lambda}\right)^{\oplus c_{\mu}}.
\]
\end{prop}

\subsubsection{Finding the $S_{2k}\times S_{k}$ representation generated by $\protect\normxilambda$\label{SUBsubsec: PROPERTIES OF XI_=00005Clambda}}

We will use the Gelfand--Tsetlin basis of $\mathcal{U}_{\lambda^{+}(n)}$
to show that $\normxilambda\in\mathcal{U}_{\lambda^{+}(2k)},$ which
leads to Proposition \ref{prop: V^=00005Clambda^+(2k) =00005Cotimes V^=00005Clambda generated by xi}.
We need to show that $\normxilambda$ is orthogonal to the subspace
$\mathcal{U}_{\lambda^{+}(2k)}^{\perp}.$ Consider the subgroup

\[
\mathcal{D}=\big\langle s_{1},\ s_{3},\dots,\ s_{2k-1}\big\rangle\leq S_{2k}\leq S_{n},
\]
which appeared in \S\ref{subsec:HYPEROCTAHEDRAL GROUP} and notice
that, for each $i=1,3,\dots,2k-1$, 
\begin{equation}
\rho(s_{i})\left(\normxilambda\right)=-\normxilambda,\label{eq: D acts by determinant}
\end{equation}
so that 
\[
\Delta\left(s_{i},\mathrm{Id}\right)\left(\normxilambda\right)=-\normxilambda.
\]
We write $\normxilambda$ in the Gelfand--Tsetlin basis: 
\begin{equation}
\normxilambda=\sum_{T_{1},T_{2}}\beta_{T_{1},T_{2}}\left(v_{T_{1}}\otimes v_{T_{2}}\right),\label{eq: normxi lambda basis in young tableau}
\end{equation}
where 
\[
T_{1}\in\mathrm{Tab}\left(\lambda^{+}(2k)\right)\mathrm{\ and}\ T_{2}\in\mathrm{Tab}\left(\lambda\right)
\]
or $v_{T_{1}}\otimes v_{T_{2}}$ represents a Gelfand--Tsetlin basis
vector of one of the subspaces
\[
V^{\mu}\otimes V^{\lambda}
\]
where $\mu\vdash2k$ and $\mu_{1}>k.$ In this case we will have 
\[
T_{1}\in\bigcup_{\mu\vdash2k,\ \mu_{1}>k}\mathrm{Tab}(\mu)
\]
and 
\[
T_{2}\in\mathrm{Tab}(\lambda).
\]

\begin{lem}
\label{lem: coefficient of v_T1 =00005Cotimes v_T2 of shape =00005CMU is zero}For
all $\mu\vdash2k$ with $\mu_{1}>k$ and for any $\tilde{T}_{2}\in\mathrm{Tab}(\lambda),$
if $\tilde{T}_{1}\in\mathrm{Tab}(\mu)$ then, in (\ref{eq: normxi lambda basis in young tableau}),
we have 
\[
\beta_{\tilde{T}_{1},\tilde{T}_{2}}=0.
\]
\end{lem}

\begin{proof}
Fix any $\tilde{T}_{2}\in\mathrm{Tab}(\lambda)$ and consider any
$v_{\tilde{T}_{1}}\otimes v_{\tilde{T}_{2}}$ where 
\[
\tilde{T}_{1}\in\bigcup_{\mu\vdash2k,\ \mu_{1}>k}\mathrm{Tab}(\mu).
\]
 Since $\mu_{1}>k,$ out of each of the $k$ pairs $\{1,2\},\ \{3,4\},\ \dots,\ \{2k-1,2k\},$
there must be \emph{at least one }pair, say $i$ and $i+1$, in which
both elements appear in the first row of boxes of $\tilde{T}_{1}.$
So by Proposition \ref{prop:vershik okounkov basis of V^=00005Clambda},
\begin{equation}
(s_{i},\mathrm{Id})\left(v_{\tilde{T}_{1}}\otimes v_{\tilde{T}_{2}}\right)=v_{\tilde{T}_{1}}\otimes v_{\tilde{T}_{2}}.\label{eq: (si, Id) acts trivially on v_T1 =00005Cotimesv_T2}
\end{equation}
We look at the coefficient of $v_{\tilde{T}_{1}}\otimes v_{\tilde{T}_{2}}$
in 
\[
(s_{i},\mathrm{Id})\sum_{T_{1},T_{2}}\beta_{T_{1},T_{2}}\left(v_{T_{1}}\otimes v_{T_{2}}\right).
\]
If $T_{2}\neq\tilde{T}_{2},$ then, for any choice of $T_{1}$,
\[
\begin{aligned} & \big\langle\left(s_{i},\mathrm{Id}\right)\left(v_{T_{1}}\otimes v_{T_{2}}\right),\ v_{\tilde{T}_{1}}\otimes v_{\tilde{T}_{2}}\big\rangle\\
= & \left\langle \left(s_{i}v_{T_{1}}\right),v_{\tilde{T}_{1}}\right\rangle \underbrace{\left\langle v_{T_{2}},v_{\tilde{T}_{2}}\right\rangle }_{=0}\\
= & \ 0.
\end{aligned}
\]
So now suppose that $T_{2}=\tilde{T}_{2}$ and let 
\[
T_{1}\in\mathrm{Tab}(\lambda^{+}(2k))\cup\bigcup_{\mu\vdash2k,\ \mu_{1}>k}\mathrm{Tab}(\mu),
\]
with $T_{1}\neq\tilde{T}_{1}.$ If $T_{1}$ is not of the same shape
as $\tilde{T}_{1},$ then clearly
\[
\Big\langle s_{i}v_{T_{1}},v_{\tilde{T}_{1}}\Big\rangle=0,
\]
which implies 
\[
\big\langle(s_{i},\mathrm{Id})\left(v_{T_{1}}\otimes v_{\tilde{T}_{2}}\right),\ v_{\tilde{T}_{1}}\otimes v_{\tilde{T}_{2}}\big\rangle=0.
\]
If $T_{1}$ and $\tilde{T}_{1}$ are of the same shape, then there
are $3$ possibilities for the positions of the boxes labeled $i$
and $i+1$ in $T_{1}$: 
\begin{enumerate}
\item If the boxes labeled $i$ and $i+1$ are in the same row of $T_{1},$
then $s_{i}v_{T_{1}}=v_{T_{1}}.$
\item If the boxes labeled $i$ and $i+1$ are in the same column of $T_{1}$
then $s_{i}v_{T_{1}}=-v_{T_{1}}.$
\item If the boxes labeled $i$ and $i+1$ are in neither the same row or
column then we have $s_{i}v_{T_{1}}=\left(r^{-1}v_{T_{1}}\right)+\left(\sqrt{1-r^{-2}}\right)v_{T_{1}'},$
where $r$ and $T_{1}'$ are as defined as in Proposition \ref{prop:vershik okounkov basis of V^=00005Clambda}. 
\end{enumerate}
In any of the above cases, we have 
\[
\Big\langle s_{i}v_{T_{1}},v_{\tilde{T}_{1}}\Big\rangle=0.
\]
This is because $T_{1}\neq\tilde{T}_{1}$ and, in the final case,
we also have $T_{1}'\neq\tilde{T}_{1}$. The above observations imply
\[
\begin{aligned} & \ \left\langle \Delta\left(s_{i},\mathrm{Id}\right)\normxilambda,\ v_{\tilde{T}_{1}}\otimes v_{\tilde{T}_{2}}\right\rangle \\
= & \ \sum_{T_{1},T_{2}}\beta_{T_{1},T_{2}}\left\langle \left(s_{i},\mathrm{Id}\right)\left(v_{T_{1}}\otimes v_{T_{2}}\right),\ v_{\tilde{T}_{1}}\otimes v_{\tilde{T}_{2}}\right\rangle \\
= & \ \beta_{\tilde{T}_{1},\tilde{T}_{2}}\left\langle \left(s_{i},\mathrm{Id}\right)\left(v_{\tilde{T}_{1}}\otimes v_{\tilde{T}_{2}}\right),\ v_{\tilde{T}_{1}}\otimes v_{\tilde{T}_{2}}\right\rangle .
\end{aligned}
\]
Using (\ref{eq: (si, Id) acts trivially on v_T1 =00005Cotimesv_T2}),
this is exactly
\[
\beta_{\tilde{T}_{1},\tilde{T}_{2}}.
\]
 But (\ref{eq: D acts by determinant}) implies that the coefficient
of $v_{\tilde{T}_{1}}\otimes v_{\tilde{T}_{2}}$ in $\Delta(s_{i},\mathrm{Id})\normxilambda$
is $-\beta_{\tilde{T}_{1},\tilde{T}_{2}}$, implying that $\beta_{\tilde{T}_{1},\tilde{T}_{2}}=0.$ 
\end{proof}
\begin{prop}
\label{prop: V^=00005Clambda^+(2k) =00005Cotimes V^=00005Clambda generated by xi}
When $n\geq2k,$ given any $\lambda\vdash k,$ we have 
\[
\Big\langle\Delta(g,\sigma)\left(\normxilambda\right):\ g\in S_{2k},\ \sigma\in S_{k}\Big\rangle_{\C}=\mathcal{U}_{\lambda^{+}(2k)}.
\]
\end{prop}

\begin{proof}
This follows from Proposition \ref{prop:isomorphism of S2k x Sk module generated by xi - general partitions }
and Lemma \ref{lem: coefficient of v_T1 =00005Cotimes v_T2 of shape =00005CMU is zero},
since Lemma \ref{lem: coefficient of v_T1 =00005Cotimes v_T2 of shape =00005CMU is zero}
implies that 
\[
\normxilambda\in\mathcal{U}_{\lambda^{+}(2k)}.
\]
 
\end{proof}

\subsubsection{Constructing $W^{\emptyset,\lambda}\otimes V^{\lambda}$ inside $\protect\cnk$\label{subsec:CONSTRUCTING W^=00005Cphi,=00005Clambda =00005Cotimes V^=00005Clambda INSIDE CNK}}

We will show that $\normxilambda$ is in the $\left(W^{\emptyset,\lambda}\otimes V^{\lambda}\right)$-isotypic
subrepresentation of $\mathcal{U}_{\lambda^{+}(2k)}\downarrow_{H_{k}\times S_{k}}$
and that this subrepresentation has multiplicity one, so that 
\[
\Big\langle\Delta(g,\sigma)\left(\normxilambda\right):g\in H_{k},\ \sigma\in S_{k}\Big\rangle_{\C}\cong W^{\emptyset,\lambda}\otimes V^{\lambda}.
\]
We write the decomposition 
\begin{equation}
\mathcal{U}_{\lambda^{+}(2k)}\downarrow_{H_{k}\times S_{k}}\cong\bigoplus_{(\mu,\pi)\models2k,\nu\vdash k}\left(W^{\mu,\pi}\otimes V^{\nu}\right)^{\oplus c(\mu,\pi,\nu)}.\label{eq: res V^lambda2k =00005CotimesV^lambda}
\end{equation}
By definition, $\normxilambda$ must be orthogonal to any component
of this decomposition with $\nu\neq\lambda.$ Moreover, by (\ref{eq: D acts by determinant}),
for any generator $s_{i}$ of $\mathcal{D},$ the element $\left(s_{i},\mathrm{Id}\right)\in H_{k}\times S_{k}$
and $\Delta(s_{i},\mathrm{Id})\left(\normxilambda\right)=-\normxilambda.$
The only irreducible representations of $H_{k}\times S_{k}$ for which
this property holds \emph{for every generator $s_{i}$ }are representations
of the form 
\[
W^{\emptyset,\pi}\otimes V^{\nu},
\]
where $\pi\vdash k$ and $\nu\vdash k$. Combining these observations
implies that $\normxilambda$ is orthogonal to any subrepresentation
that is \emph{not} isomorphic to 
\[
\left(W^{\emptyset,\pi}\otimes V^{\lambda}\right)^{\oplus c(\emptyset,\pi,\lambda)},
\]
where $\pi\vdash k$. With the observation that for any $\sigma\in S_{k},$
\[
\theta(\sigma)\left(\normxilambda\right)=\rho\left(\psi\left(\sigma^{-1}\right)\right)\left(\normxilambda\right),
\]
it follows that $\normxilambda$ belongs to the $\left(W^{\emptyset,\lambda}\otimes V^{\lambda}\right)$-isotypic
component in the decomposition. It remains to show that this isotypic
component has multiplicity one. A first attempt would be to use the
branching rule given by Koike and Terada in \cite{KoikeTerada1987}.
\begin{prop}
\label{prop:branching rule koike and terada}Denote by $s_{\mu}$
a Schur polynomial, $e_{i}$ the elementary symmetric polynomial of
degree $i$ and $h_{j}$ the $j^{\mathrm{th}}$complete symmetric
polynomial. Then, given $\pi\vdash2k,$ 
\[
\mathrm{Res}_{H_{k}}^{S_{2k}}V^{\pi}\cong\bigoplus_{(\mu,\nu)\models k}\left(W^{\mu,\nu}\right)^{\oplus d_{\mu,\nu}^{\pi}},
\]
 where the multiplicity $d_{\mu,\nu}^{\pi}$ of each $W^{\mu,\nu}$
coincides exactly with the coefficient of $s_{\pi}$ in the product
$\left(s_{\mu}\circ h_{2}\right)\left(s_{\nu}\circ e_{2}\right)$.
That is, $d_{\mu,\nu}^{\pi}$ satisfies
\[
\left(s_{\mu}\circ h_{2}\right)\left(s_{\nu}\circ e_{2}\right)=\sum_{\pi}d_{\mu,\nu}^{\pi}s_{\pi}.
\]
\end{prop}

Following this proposition, we write
\begin{equation}
\mathcal{U}_{\lambda^{+}(2k)}\downarrow_{H_{k}\times S_{k}}\cong\bigoplus_{(\mu,\nu)\models k}\left(W^{\mu,\nu}\otimes V^{\lambda}\right)^{\oplus d_{\mu,\nu}^{\lambda^{+}(2k)}}\label{eq:decomp of restriction V^=00005Clambda+2k =00005CotimesV=00005ClambdageneralHkirreps}
\end{equation}
and the following lemma is immediate from the preceding discussion.
\begin{lem}
\label{prop:H_k x S_k rep generated by xi is W =00005Cotimes V isotypic with multiplicity}With
$d_{\mu,\nu}^{\pi}$ defined as in Proposition \ref{prop:branching rule koike and terada}
, for some $\kappa_{\emptyset,\lambda}^{\lambda^{+}(2k)}$ satisfying
$1\leq\kappa_{\emptyset,\lambda}^{\lambda^{+}(2k)}\leq d_{\emptyset,\lambda}^{\lambda^{+}(2k)},$
\begin{equation}
\Big\langle\Delta(g,\sigma)\left(\normxilambda\right):\ g\in H_{k},\ \sigma\in S_{k}\Big\rangle_{\C}\cong\left(W^{\emptyset,\lambda}\otimes V^{\lambda}\right)^{\oplus\kappa_{\emptyset,\lambda}^{\lambda^{+}(2k)}},\label{eq: decomp of span of Hk x Sk on xi is W =00005CotimesV^=00005Clmabdaisotypicwithmultiplicity}
\end{equation}
 as representations of $H_{k}\times S_{k}.$
\end{lem}

The aim is to prove Proposition \ref{prop: restrict V^l=00005Cmabda^+(2k) to H_k},
restated below.
\begin{prop}
\label{prop:multiplicity of W =00005Cphi,=00005Clambda is exactly one}The
multiplicity $d_{\emptyset,\lambda}^{\lambda^{+}(2k)}$ of $W^{\emptyset,\lambda}$
in the restriction of $V^{\lambda^{+}(2k)}$ from $S_{2k}$ to $H_{k}$
is exactly $1.$
\end{prop}

We do not know how to use Proposition \ref{prop:branching rule koike and terada}
effectively, since, in general, it is difficult to compute 
\[
\left(s_{\emptyset}\circ h_{2}\right)\left(s_{\lambda}\circ e_{2}\right).
\]
Using that $s_{\emptyset}$ is the constant function and $e_{2}=s_{(1,1)},$
our task reduces to showing that the coefficient of $s_{\lambda^{+}(2k)}$
in the Schur polynomial expansion of 
\begin{equation}
s_{\lambda}\circ s_{(1,1)}\label{eq: schur version of the restriction of s2k to hk}
\end{equation}
is indeed one$.$ This plethysm can be evaluated in the special cases
whereby $|\lambda|\leq3,$ see for example \cite[Theorem 5.3]{Colmenarejo+}
and in some other special cases of $\lambda\vdash k$ (for example,
when $\lambda=(k)$ or $\lambda=(1,\dots,1)$, see \cite[Section 5.3]{Colmenarejo+}).
To our knowledge, there is no simple expression for (\ref{eq: schur version of the restriction of s2k to hk})
for every $\lambda,$ and even the task at hand of evaluating just
one specific coefficient in the Schur expansion does not appear to
have an obvious straightforward approach. 

Instead, we will construct the $W^{\emptyset,\lambda}$-isotypic subspace
of $V^{\lambda^{+}(2k)}\downarrow_{H_{k}}$ using the Gelfand--Tsetlin
basis and, in doing so, we will see that the multiplicity must be
one. Recall that the character of $W^{\emptyset,\lambda}$ is $\chi^{\lambda}\chi_{\rho_{0}},$
so that the generators of $\mathcal{D}$ act on $W^{\emptyset,\lambda}$
by multiplying elements by $-1$. Which is to say that the $W^{\emptyset,\lambda}$-isotypic
subspace of $V^{\lambda^{+}(2k)}\downarrow_{H_{k}}$ must be contained
in the sign-isotypic component of $\mathcal{D}$ in the vector space
$V^{\lambda^{+}(2k)}$. Denote by
\[
\signisotypic\overset{\mathrm{def}}{=}\mathrm{sign\text{-}isotypic\ component\ of}\ \mathcal{D}\mathrm{\ in}\ V^{\lambda^{+}(2k)},
\]
so that
\begin{equation}
\left(W^{\emptyset,\lambda}\right)^{d_{\emptyset,\lambda}^{\lambda^{+}(2k)}}\subseteq\signisotypic.\label{eq: W^=00005Cphi,=00005Clambdaiscontainedinsidesignisotypiccomp}
\end{equation}
Our method is as follows:
\begin{enumerate}
\item We will show that $\signisotypic$ is a $H_{k}$-submodule of $V^{\lambda^{+}(2k)}\downarrow_{H_{k}}$;
\item Then we will construct $\signisotypic$ using the Gelfand--Tsetlin
basis of $V^{\lambda^{+}(2k)}$;
\item In doing the above step, we will see that $\signisotypic$ has dimension
$d_{\lambda},$ so that (\ref{eq: W^=00005Cphi,=00005Clambdaiscontainedinsidesignisotypiccomp})
implies that $\signisotypic=W^{\emptyset,\lambda}$ and that $d_{\emptyset,\lambda}^{\lambda^{+}(2k)}=1.$
\end{enumerate}
Define
\begin{equation}
\begin{aligned} & \mathrm{Tab}_{\mathcal{\mathcal{D}},\mathrm{sign}}\left(\lambda^{+}(2k)\right)\\
\overset{\mathrm{def}}{=} & \Big\{ T\in\mathrm{Tab}\left(\lambda^{+}(2k)\right)\ : & T\mathrm{\ has\ exactly\ one\ representative\ from\ each\ of\ \{1,2\},}\\
 &  & \mathrm{\{3,4\},\dots,\ \{2k-1,2k\}\ in\ the\ first\ row\ of\ boxes\Big\}}
\end{aligned}
\label{eq: definition of Tab_D subset of vectors}
\end{equation}
and the associated subspace
\[
\left\langle v_{\hat{T}}:\ \hat{T}\in\tablambda\right\rangle \subseteq V^{\lambda^{+}(2k)}.
\]

\begin{prop}
\label{prop:sign isotypic component of D on V^=00005Clambda+2k is contained in TabD}
For any $\lambda\vdash k,$
\[
\signisotypic\subseteq\left\langle v_{\hat{T}}:\ \hat{T}\in\tablambda\right\rangle .
\]
\end{prop}

\begin{proof}
Let 
\[
\hat{T}\in\mathrm{Tab}\left(\lambda^{+}(2k)\right)\backslash\tablambda
\]
 and suppose, towards a contradiction, that $u\in\signisotypic$ is
such that 
\[
\left\langle u,v_{\hat{T}}\right\rangle =\alpha\neq0.
\]
Since $\lambda^{+}(2k)$ has $k$ boxes in the top row and $\hat{T}\notin\tablambda,$
there must be some $i\in\{1,\ 3,\dots,\ 2k-1\}$ for which the boxes
labeled $i$ and $i+1$ are both in the top row. Then $s_{i}=(i\ i+1)$
acts trivially on $v_{\hat{T}}$ by Proposition \ref{prop:vershik okounkov basis of V^=00005Clambda}.
Using the same reasoning as Lemma \ref{lem: coefficient of v_T1 =00005Cotimes v_T2 of shape =00005CMU is zero},
if $\hat{T}_{1}\neq\hat{T},$ then $\left\langle s_{i}v_{\hat{T}_{1}},v_{\hat{T}}\right\rangle =0$.
So then 
\[
\begin{aligned}\left\langle s_{i}u,v_{\hat{T}}\right\rangle  & =\left\langle u,v_{\hat{T}}\right\rangle \left\langle s_{i}v_{\hat{T}},v_{\hat{T}}\right\rangle \\
 & =\left\langle u,v_{\hat{T}}\right\rangle \\
 & =\alpha.
\end{aligned}
\]
However, since $u\in\signisotypic$ and $s_{i}\in\mathcal{D},$ we
must have 
\[
s_{i}u=-u,
\]
which obviously implies that 
\[
\left\langle s_{i}u,v_{\hat{T}}\right\rangle =-\alpha,
\]
contradicting the fact that $\alpha\neq0.$ 
\end{proof}
\begin{prop}
\label{prop:Tsign isotypic com is a H_k submodule}For each $\lambda\vdash k$,
the space $\signisotypic\subset V^{\lambda^{+}(2k)}$ is a $H_{k}$-submodule.
\end{prop}

\begin{proof}
Let $h\in H_{k},$ $u\in\signisotypic$ and $\delta\in\D.$ Since
$\mathcal{D}$ is normal in $H_{k},$ we have 
\[
\begin{aligned}\delta(hu) & =hh^{-1}\delta(hu)\\
 & =h(h^{-1}\delta h)u\\
 & =\mathrm{sign}\left(h^{-1}\delta h\right)(hu)\\
 & =\mathrm{sign}(\delta)(hu),
\end{aligned}
\]
so that $hu\in\signisotypic.$ 
\end{proof}
For any $\mu$ of any given size, say $\mu\vdash l$, we define an
injective map 
\[
\Psi:V^{\mu}\to V^{\mu^{+}(2l)},
\]
where $\Psi(v_{T})$ corresponds to the Young tableau of shape $\mu^{+}(2l)$,
in which the labels of the boxes in the positions of $\mu^{+}(2l)^{*}$
are double the corresponding label in $T,$ and the top row of boxes
contains the labels (in order) $1,3,\dots,2l-1.$ By a slight abuse
of notation, we will label this Young tableau by $\Psi(T)$, so that
$\Psi\left(v_{T}\right)=v_{\Psi(T)}$. For example, with $\mu=(2,1),$

\begin{figure}[H]
\centering
\begin{ytableau} 1&2 \cr 3 \end{ytableau}$\ \overset{\Psi}{\mapsto}\ $\begin{ytableau} 1&3&5 \cr 2&4 \cr 6 \end{ytableau}
and \begin{ytableau} 1&3 \cr 2 \end{ytableau}$\ \overset{\Psi}{\mapsto}\ $\begin{ytableau} 1&3&5 \cr 2&6 \cr 4 \end{ytableau}

\end{figure}

For each $\lambda\vdash k$ and for each $T\in\mathrm{Tab}(\lambda),$
define 
\begin{equation}
Z_{T}\overset{\mathrm{def}}{=}\left\langle \D\Psi\left(v_{T}\right)\right\rangle \subseteq V^{\lambda^{+}(2k)}.\label{eq: definition of Z_T}
\end{equation}
We must establish some extra notation for the remainder of this section.
For a given Young tableau $T$ of any shape, if the boxes labeled
$i$ and $i+1$ are in different rows and columns, then we will now
denote by $T'_{\{i,i+1\}}$ the Young tableau of the same shape obtained
by swapping the labels $i$ and $i+1.$ If we also have $j\neq i,i+1$
and $j+1\neq i,i+1$, with the boxes labeled $j$ and $j+1$ in different
rows and columns, then we denote by $T''_{\{i,i+1\},\{j,j+1\}}$ the
Young tableau of the same shape obtained by swapping the labels $i$
and $i+1$ and then the labels $j$ and $j+1.$ The Young tableau
\[
T_{\{i_{1},i_{1}+1\},\dots,\{i_{m},i_{m}+1\}}^{(m)}
\]
is defined in the obvious way. If $T$ is a Young tableau of shape
$\mu\vdash l,$ we will write $T\backslash\{l\}$ for the Young tableau
of shape $\mu'\vdash\left(l-1\right)$, the YD obtained from $\mu$
by removing the box labeled by $l$ in $T$ and keeping all other
boxes and labels the same. We define $T\backslash\{l,l-1,\dots,l-j\}$
inductively in the obvious way. 
\begin{lem}
\label{lem: Z_T are orthogonal}Let $\lambda\vdash k$ and $T,\dot{T}\in\mathrm{Tab}(\lambda).$
Then, if $T\neq\dot{T},$ $Z_{T}\perp Z_{\dot{T}}.$
\end{lem}

\begin{proof}
By Proposition \ref{prop:vershik okounkov basis of V^=00005Clambda},
any $v\in Z_{T}$ must be written as a linear combination of $v_{\tilde{T}}$
such that, if a box $\square\in\lambda$ has label $x\in[k]$ in $T$,
then the corresponding box $\tilde{\square}\in\left(\lambda^{+}(2k)\right)^{*}$
has label either $2x$ or $2x-1$ in $\tilde{T}.$ The same is true
for any element $u\in Z_{\dot{T}}.$ Since $T\neq\dot{T,}$ there
is a box $\square\in\lambda$ labeled $x$ in $T$ and $y$ in $\dot{T}$
with $x\neq y.$ So any $v\in Z_{T}$ is written as a linear combination
of $v_{\tilde{T}}$ as described and then, since $x\neq y,$ we cannot
have $v\in Z_{\dot{T}}.$
\end{proof}
\begin{lem}
\label{lem: Tab_D,sign is contained in Z_T span}For every $\lambda\vdash k$,
\[
\left\langle v_{\hat{T}}:\ \hat{T}\in\tablambda\right\rangle \subseteq\bigoplus_{T\in\mathrm{Tab}(\lambda)}Z_{T}.
\]
\end{lem}

\begin{proof}
We will prove this (for each basis vector of the LHS) by induction
on the number of \emph{even }elements in the top row of $\hat{T}\in\tablambda.$
Denote by 
\[
q_{E,\hat{T}}
\]
the number of boxes with an even label in the top row of any Young
tableau $\hat{T}\in\mathrm{Tab}\left(\lambda^{+}(2k)\right)$ and
suppose that 
\[
\hat{T}\in\tablambda
\]
is such that $q_{E,\hat{T}}=0.$ Then, in $\hat{T}$, all of the boxes
in $\left(\lambda^{+}(2k)\right)^{*}$ \emph{must }have an even label.
Let $T^{*}\in\mathrm{Tab}(\lambda)$ be such that
\[
\Psi\left(v_{T^{*}}\right)=v_{\hat{T}},
\]
which implies that 
\[
v_{\hat{T}}\in Z_{T^{*}}\subseteq\bigoplus_{T\in\mathrm{Tab}(\lambda)}Z_{T}.
\]

Now suppose that 
\[
\hat{T}\in\tablambda
\]
 is such that $q_{E,\hat{T}}>0$ and that, for any $T\in\tablambda$
with $0\leq q_{E,T}<q_{E,\hat{T}},$ we have $v_{T}\in\bigoplus_{T\in\mathrm{Tab}(\lambda)}Z_{T}.$
Suppose that the first \emph{even} label in the top row of boxes of
$\hat{T}$ is $j$ and let $\tilde{T}=\hat{T}'_{\{j-1,j\}}.$ Then
we also have 
\[
\tilde{T}\in\tablambda
\]
so that
\[
v_{\tilde{T}}\in\left\langle v_{\hat{T}}:\ \hat{T}\in\tablambda\right\rangle .
\]
Moreover, $q_{E,\tilde{T}}=q_{E,\hat{T}}-1,$ so, by the inductive
hypothesis, $v_{\tilde{T}}\in\bigoplus_{T\in\mathrm{Tab}(\lambda)}Z_{T}.$ 

Since the $Z_{T}$ are pairwise orthogonal, there is exactly one $T^{*}\in\mathrm{Tab}(\lambda)$
such that $v_{\tilde{T}}\in Z_{T^{*}}.$ By Proposition \ref{prop:vershik okounkov basis of V^=00005Clambda},
we have 
\[
s_{j-1}v_{\tilde{T}}=\left(r^{-1}\right)v_{\tilde{T}}+\left(\sqrt{1-r^{-2}}\right)v_{\tilde{T}'_{\{j-1,j\}}},
\]
 where $r$ is as defined in the proposition. This shows that
\[
\left(r^{-1}\right)v_{\tilde{T}}+\left(\sqrt{1-r^{-2}}\right)v_{\tilde{T}'_{\{j-1,j]}}\in Z_{T^{*}},
\]
which implies that 
\[
v_{\tilde{T}'_{\{j-1,j]}}\in Z_{T^{*}}.
\]
But $\tilde{T}'_{\{j-1,j\}}=\hat{T}$, so we have $v_{\hat{T}}\in Z_{T^{*}}\subseteq\bigoplus_{T\in\mathrm{Tab}(\lambda)}Z_{T}.$ 
\end{proof}
\begin{lem}
\label{lem:sign isotypic subspace of Z_Ts is same as that of V^=00005Clambda2k}For
any $\lambda\vdash k$, the sign-isotypic component of $\mathcal{D}$
in $\bigoplus_{T\in\mathrm{Tab}(\lambda)}Z_{T}$ is exactly $\signisotypic.$
\end{lem}

\begin{proof}
Since $\bigoplus_{T\in\mathrm{Tab}(\lambda)}Z_{T}\subseteq V^{\lambda^{+}(2k)},$
the sign-isotypic component of $\mathcal{D}$ in 
\[
\bigoplus_{T\in\mathrm{Tab}(\lambda)}Z_{T}
\]
 is obviously contained in $\signisotypic.$ On the other hand, by
Proposition \ref{prop:sign isotypic component of D on V^=00005Clambda+2k is contained in TabD},
$\signisotypic$ must be contained in the sign-isotypic component
of $\mathcal{D}$ in 
\[
\left\langle v_{T}:\ T\in\tablambda\right\rangle ,
\]
 which, by the previous lemma, must be contained inside the sign-isotypic
component of $\mathcal{D}$ in $\bigoplus_{T\in\mathrm{Tab}(\lambda)}Z_{T}.$
\end{proof}
For each $T\in\mathrm{Tab}(\lambda),$ we denote by $Z_{T}^{\mathcal{D},\mathrm{sign}}$
the sign-isotypic subspace of $\mathcal{D}$ in $Z_{T}.$
\begin{prop}
For any $\lambda\vdash k$ and for any $T\in\mathrm{Tab}(\lambda)$,
$Z_{T}^{\mathcal{D},\mathrm{sign}}$ is one dimensional. \label{prop:Z_T has a one dimensional subspace on which D acts by sign}
\end{prop}

To prove Proposition \ref{prop:Z_T has a one dimensional subspace on which D acts by sign}
we introduce some additional notation and prove three intermediate
lemmas. For each $\lambda\vdash k$ and for each $T\in\mathrm{Tab}\left(\lambda\right)$,
for each $i\in\{1,\dots,k-1\},$ define 
\[
Z_{T}^{(i)}\overset{\mathrm{def}}{=}\left\langle \D^{(i)}\Psi\left(v_{T\backslash\{k,\dots,k-i+1\}}\right)\right\rangle ,
\]
where 
\[
\D^{(i)}\overset{\mathrm{def}}{=}\left\langle s_{1},\dots,s_{2k-2i-1}\right\rangle .
\]
Define 
\[
\D_{\perp}^{(i)}\overset{\mathrm{def}}{=}\left\langle s_{2k-2i+1},\dots,s_{2k-1}\right\rangle 
\]
and also 
\[
\mathrm{Tab}\left(Z_{T}^{(i)}\right)=\Big\{\hat{T}\in\mathrm{Tab}(\mu):\mu\vdash(2k-2i),\ \left\langle z,v_{\hat{T}}\right\rangle \neq0\ \mathrm{for\ some}\ z\in Z_{T}^{(i)}\Big\}.
\]
 Then, for each $i\in\{1,\dots,k\},$ we introduce a map 
\[
\Phi_{T,i}^{\lambda}:Z_{T}^{(i-1)}\to\bigoplus_{\mu\vdash2k-2i}V^{\mu}
\]
whereby 
\[
\Phi_{T,i}^{\lambda}\left(v_{\hat{T}}\right)=v_{\hat{T}\backslash\{2k-2i+2,2k-2i+1\}}
\]
for any $\hat{T}\in\mathrm{Tab}\left(Z_{T}^{(i)}\right)$ and $Z_{T}^{(0)}$
is understood to be $Z_{T}.$ 
\begin{lem}
\label{lem: PHI is homomorphis of module}For any $\lambda\vdash k,$
$T\in\mathrm{Tab}(\lambda)$ and for any $i\in\{1,\dots,k-1\},$ the
map $\Phi_{T,i}^{\lambda}$ is a $\D^{(i)}$-module homomorphism.
\end{lem}

\begin{proof}
To prove this lemma, we need to show that $\Phi_{T,i}^{\lambda}\left(s_{j}v_{\hat{T}}\right)=s_{j}\Phi_{T,i}^{\lambda}\left(v_{\hat{T}}\right)$
for any $\hat{T}\in\mathrm{Tab}\left(Z_{T}^{(i-1)}\right)$ and for
any $j\in\{1,3,\dots,2k-2i-1\}.$ To this end, fix any such $\hat{T}$
and $j.$ 

There are three possibilities for the positions of the boxes labeled
$j$ and $j+1$ in $\hat{T}.$ If they are in the same row/same column/neither
the same row or column, then this is the exact same relationship between
the boxes labeled $j$ and $j+1$ in $\hat{T}\backslash\{2k-2i+2,2k-2i+1\}.$
In the first case we have 
\[
\Phi_{T,i}^{\lambda}\left(s_{j}v_{\hat{T}}\right)=\Phi_{T,i}^{\lambda}\left(v_{\hat{T}}\right)=s_{j}\Phi_{T,i}^{\lambda}\left(v_{\hat{T}}\right)
\]
and, in the second case, 
\[
\Phi_{T,i}^{\lambda}\left(s_{j}v_{\hat{T}}\right)=\Phi_{T,i}^{\lambda}\left(-v_{\hat{T}}\right)=-\Phi_{T,i}^{\lambda}\left(v_{\hat{T}}\right)=s_{j}\Phi_{T,i}^{\lambda}\left(v_{\hat{T}}\right).
\]
In the case where these two boxes are in neither the same row or column,
the contents of the boxes labeled $j$ and $j+1$ are always the same
in both $\hat{T}$ and $\hat{T}\backslash\{2k-2i,2k-2i-1\}.$ Also
in this case, 
\[
\Phi_{T,i}^{\lambda}\left(v_{\hat{T}'_{\{j,j+1\}}}\right)=v_{\left(\hat{T}\backslash\{2k-2i+2,2k-2i+1\}\right)'_{\{j,j+1\}}}.
\]
So then we have 
\[
\begin{aligned}\Phi_{T,i}^{\lambda}\Big(s_{j}v_{\hat{T}}\Big) & =\Phi_{T,i}^{\lambda}\Big(\left(r^{-1}\right)v_{\hat{T}}+\left(\sqrt{1-r^{-2}}\right)v_{\hat{T}'_{\{j,j+1\}}}\Big)\\
 & =\left(r^{-1}\right)\Phi_{T,i}^{\lambda}\left(v_{\hat{T}}\right)+\left(\sqrt{1-r^{-2}}\right)\Phi_{T,i}^{\lambda}\left(v_{\hat{T}'_{\{j,j+1\}}}\right)\\
 & =\left(r^{-1}\right)v_{\hat{T}\backslash\{2k-2i+2,2k-2i+1\}}+\left(\sqrt{1-r^{-2}}\right)v_{\left(\hat{T}\backslash\{2k-2i+2,2k-2i+1\}\right)'_{\{j,j+1\}}}\\
 & =s_{j}v_{\hat{T}\backslash\{2k-2i+2,2k-2i+1\}}\\
 & =s_{j}\Phi_{T,i}^{\lambda}\left(v_{\hat{T}}\right).
\end{aligned}
\]
 
\end{proof}
The next lemma asserts that $\Phi_{T,i}^{\lambda}$ is injective on
the $-1$ eigenspace of $s_{2k-2i+1}$ in $Z_{T}^{(i-1)},$ so that
$\Phi_{T,i}^{\lambda}$ always defines a $\D^{(i)}$-module isomorphism
between this eigenspace and its image.
\begin{lem}
\label{lem:PHI is injective on eigenspace}For any $\lambda\vdash k,$
$T\in\mathrm{Tab}(\lambda)$ and for any $i\in\{1,\dots,k\},$ the
map $\Phi_{T,i}^{\lambda}$ is injective on the $-1$ eigenspace of
$s_{2k-2i+1}$ in $Z_{T}^{(i-1)}.$ 
\end{lem}

\begin{proof}
It is easy to see that 
\[
\ker\left(\Phi_{T,i}^{\lambda}\right)=\bigoplus_{\hat{T}}\left\langle v_{\hat{T}}-v_{\hat{T}'_{\{2k-2i+2,2k-2i+1\}}}\right\rangle ,
\]
where the direct sum is over all $\hat{T}\in\mathrm{Tab}\left(Z_{T}^{(i-1)}\right)$
for which $\hat{T}'_{\{2k-2i+2,2k-2i+1\}}$ is defined (i.e. still
a valid Young tableaux). The $-1$ eigenspace of $s_{2k-2i+1}$ in
$Z_{T}^{(i-1)}$ is exactly 
\[
\bigoplus_{\hat{T}_{1}}\left\langle v_{\hat{T}_{1}}\right\rangle \oplus\bigoplus_{\hat{T}}\left\langle v_{\hat{T}}-\sqrt{\frac{1+r^{-1}}{1-r^{-1}}}v_{\hat{T}'_{\{2k-2i+2,2k-2i+1\}}}\right\rangle ,
\]
where the first direct sum is over all $\hat{T}_{1}\in\mathrm{Tab}\left(Z_{T}^{(i-1)}\right)$
for which $2k-2i+1$ and $2k-2i+2$ are in the same column and the
second direct sum is over all $\hat{T}\in\mathrm{Tab}\left(Z_{T}^{(i-1)}\right)$
for which $2k-2i+1$ and $2k-2i+2$ are in neither the same row or
the same column (the case where they are both in the same row is not
possible, since this cannot be the case in $\Psi(T)$ and in $\hat{T},$
the boxes labeled $2k-2i+1$ and $2k-2i+2$ are either in the same
boxes as in $\Psi(T)$ or they have swapped with one another). 

Clearly, this does not intersect $\ker\left(\Phi_{T,i}^{\lambda}\right)$
other than at $0$, so $\Phi_{T,i}^{\lambda}$ is injective when restricted
to the $-1$ eigenspace of $s_{2k-2i+1}$ in $Z_{T}^{(i-1)}.$ 
\end{proof}
The final intermediate lemma determines the image of $\Phi_{T,i}^{\lambda}$
when restricted to this eigenspace. 
\begin{lem}
\label{lem:image of PHI}The image of the $-1$ eigenspace of $s_{2k-2i+1}$
in $Z_{T}^{(i-1)}$ under $\Phi_{T,i}^{\lambda}$ is $Z_{T}^{(i)}.$
\end{lem}

\begin{proof}
We have 
\[
\begin{aligned} & \Phi_{T,i}^{\lambda}\left(\bigoplus_{\hat{T}_{1}}\left\langle v_{\hat{T}_{1}}\right\rangle \oplus\bigoplus_{\hat{T}}\left\langle v_{\hat{T}}-\sqrt{\frac{1+r^{-1}}{1-r^{-1}}}v_{\hat{T}'_{\{2k-2i+2,2k-2i+1\}}}\right\rangle \right)\\
= & \ \left\langle v_{\hat{T}\backslash\{2k-2i+2,2k-2i+1\}}:\ \hat{T}\in\mathrm{Tab}\left(Z_{T}^{(i-1)}\right)\right\rangle \\
= & \ \left\langle \D^{(i)}\Psi\left(v_{T\backslash\{k,\dots,k-i+1\}}\right)\right\rangle \\
= & \ Z_{T}^{(i)}.
\end{aligned}
\]
\end{proof}
\begin{proof}[Proof of Proposition \ref{prop:Z_T has a one dimensional subspace on which D acts by sign}]
 We can apply the previous three lemmas iteratively to determine
$Z_{T}^{\mathcal{D},\mathrm{sign}}.$ Applying them for $i=1$ shows
that the $-1$ eigenspace of $s_{2k-1}$ in $Z_{T}$ is isomorphic
to $Z_{T}^{(1)}.$ Repeating this for $i=2$ shows that the $-1$
eigenspace of $s_{2k-3}$ in $Z_{T}^{(1)}$ is isomorphic$Z_{T}^{(2)}.$
But the $-1$ eigenspace of $s_{2k-3}$ in $Z_{T}^{(1)}$ is the sign-isotypic
subspace of $\D_{\perp}^{(2)}$ in $Z_{T}.$ Repeating for each $j=3,,\dots,j=k-1$
shows that the sign-isotypic subspace of 
\[
\D_{\perp}^{(k-1)}=\left\langle s_{3},\dots,s_{2k-1}\right\rangle 
\]
 in $Z_{T}$ is isomorphic to 
\[
Z_{T}^{(k-1)}=\left\langle s_{1}\Psi\left(v_{T\backslash\{k,\dots,2\}}\right)\right\rangle .
\]
Thus, the sign-isotypic subspace of $\mathcal{D}$ in $Z_{T}$, $Z_{T}^{\mathcal{D},\mathrm{sign}},$
is just the $-1$ eigenspace of $s_{1}$ in $Z_{T}^{(k-1)}$. But
$Z_{T}^{(k-1)}$ is clearly the one-dimensional sign representation
of $\left\langle s_{1}\right\rangle ,$ so the $-1$ eigenspace of
$s_{1}$ is just the whole space and we conclude that $Z_{T}^{\mathcal{D},\mathrm{sign}}$
has dimension one. 
\end{proof}
By Lemma \ref{lem:sign isotypic subspace of Z_Ts is same as that of V^=00005Clambda2k},
\[
\bigoplus_{T\in\mathrm{Tab}(\lambda)}Z_{T}^{\mathcal{D},\mathrm{sign}}=\signisotypic,
\]
so that Proposition \ref{prop:Z_T has a one dimensional subspace on which D acts by sign}
implies that $\signisotypic$ has dimension $d_{\lambda}.$ This completes
the proof of Proposition \ref{prop:multiplicity of W =00005Cphi,=00005Clambda is exactly one},
which has the following corollary.
\begin{cor}
\label{cor:xi is in subspace iso to V^=00005Clambda =00005Cotimes V^=00005Clambda}For
any $\lambda\vdash k$, we have an isomorphism of representations
of $\skdelta\times S_{k}\cong S_{k}\times S_{k}$, 
\[
\mathcal{U}_{\lambda}\overset{\mathrm{def}}{=}\Big\langle\Delta(g,\sigma)\left(\normxilambda\right):\ g\in\skdelta,\ \sigma\in S_{k}\Big\rangle\cong V^{\lambda}\otimes V^{\lambda}.
\]
\end{cor}

\begin{proof}
This follows since 
\[
W^{\emptyset,\lambda}\downarrow_{S_{k}^{\psi}}\cong V^{\lambda}.
\]
\end{proof}

\subsubsection{Identifying the Projection\label{subsec:IDENTIFYING THE PROJECTION}}

Since the irreducible characters of $S_{k}$ are integer valued, we
have an isomorphism of $S_{k}$ representations, $V^{\lambda}\cong\left(V^{\lambda}\right)^{\vee}$
and subsequent isomorphisms 
\begin{equation}
V^{\lambda}\otimes V^{\lambda}\cong\left(V^{\lambda}\right)^{\vee}\otimes V^{\lambda}\cong\mathrm{End}\left(V^{\lambda}\right).\label{eq: isomorphism of V^=00005Clambda=00005CotimesV^=00005ClambdatoEnd(V^=00005Clambda)}
\end{equation}
Noting that 
\[
\psi(\sigma)\normxilambda=\normxilambda\sigma
\]
then implies that 
\[
\normxilambda\in\mathrm{End}_{S_{k}}\left(V^{\lambda}\right)\subseteq\mathrm{End}\left(V^{\lambda}\right),
\]
when $\normxilambda$ is interpreted as an element of $\mathrm{End}\left(V^{\lambda}\right)$
through the isomorphisms presented in Corollary \ref{cor:xi is in subspace iso to V^=00005Clambda =00005Cotimes V^=00005Clambda}
and (\ref{eq: isomorphism of V^=00005Clambda=00005CotimesV^=00005ClambdatoEnd(V^=00005Clambda)}).
Then, by Schur's lemma, $\normxilambda$ must correspond to a complex
scalar multiple of $\mathrm{Id}_{V^{\lambda}},$ so that 
\begin{equation}
\normxilambda=c_{\lambda}\sum_{i}w_{i}\otimes w_{i}\in\mathcal{U}_{\lambda}\subseteq\mathcal{U}_{\lambda^{+}(n)}\label{eq: xi as element of V^=00005Clambda=00005CotimesV^=00005Clambda}
\end{equation}
for some complex scalar constant $c_{\lambda}$ and any orthonormal
basis $\{w_{i}\}$ of $V^{\lambda}.$ By Lemma \ref{lem: XI has non-zero projection to each isotypic subspace},
\[
\endoofvlambda\in\left(\mathcal{U}_{\lambda^{+}(n)}\right)^{\vee}\otimes\mathcal{U}_{\lambda^{+}(n)}.
\]
There are natural isomorphisms 
\[
\left(\mathcal{U}_{\lambda^{+}(n)}\right)^{\vee}\otimes\mathcal{U}_{\lambda^{+}(n)}\cong\left(V^{\lambda^{+}(n)}\right)^{\vee}\otimes\left(V^{\lambda}\right)^{\vee}\otimes V^{\lambda^{+}(n)}\otimes V^{\lambda}\cong\left(V^{\lambda^{+}(n)}\right)^{\vee}\otimes V^{\lambda^{+}(n)}\otimes\left(V^{\lambda}\right)^{\vee}\otimes V^{\lambda},
\]
where the second isomorphism is given by permuting the tensor coordinates.
Through these isomorphisms and (\ref{eq: xi as element of V^=00005Clambda=00005CotimesV^=00005Clambda}),
we write 
\begin{equation}
\endoofvlambda=c_{\lambda}^{2}\sum_{i,j}\check{w}_{i}\otimes w_{j}\otimes\check{w}_{i}\otimes w_{j}.\label{eq: xi =00005Cotimesxicheckinbasiselements}
\end{equation}
 Let 
\[
\tilde{\mathcal{Q}}_{\lambda,n}\overset{\mathrm{def}}{=}\int_{g\in S_{n}}g\check{\xi}_{\lambda}^{\mathrm{norm}}\otimes g\normxilambda dg=\frac{1}{n!}\sum_{g\in S_{n}}g\check{\xi}_{\lambda}^{\mathrm{norm}}\otimes g\normxilambda,
\]
where $g\normxilambda=\rho(g)\left(\normxilambda\right)$ and $g\check{\xi}_{\lambda}^{\mathrm{norm}}=\rho^{*}(g)\left(\check{\xi}_{\lambda}^{\mathrm{norm}}\right)=\left(\rho(g)\left(\normxilambda\right)\right)^{\vee}.$
Then we claim that 
\[
\tilde{\mathcal{Q}}_{\lambda,n}=\frac{1}{d_{\lambda^{+}(n)}d_{\lambda}}\mathcal{Q}_{\lambda,n},
\]
where $\mathcal{Q}_{\lambda}$ is the orthogonal projection map defined
in Theorem \ref{thm: MAIN THEOREM}. Indeed, using (\ref{eq: xi =00005Cotimesxicheckinbasiselements}),
we have 
\[
\tilde{\mathcal{Q}}_{\lambda,n}=\frac{c_{\lambda}^{2}}{n!}\sum_{g\in S_{n}}\sum_{i,j}g\check{w}_{i}\otimes gw_{j}\otimes\check{w}_{i}\otimes w_{j}.
\]
Since $\tilde{\mathcal{Q}}_{\lambda,n}$ commutes with $S_{n},$ then
by Schur's lemma this corresponds to an element of 
\[
\C\mathrm{Id}_{V^{\lambda^{+}(n)}}\otimes\left(V^{\lambda}\right)^{\vee}\otimes V^{\lambda},
\]
forcing $i=j$ so that the image of $\mathcal{Q}_{\lambda,n}$ in
$\mathrm{End}\left(\mathcal{U}_{\lambda^{+}(n)}\right)$ is
\[
\frac{b_{\lambda}}{n!}\left[\mathrm{Id}_{V^{\lambda^{+}(n)}}\otimes\mathrm{Id}_{V^{\lambda}}\right],
\]
which, when we extend by zero on $\mathcal{U}_{\lambda^{+}(n)}^{\perp}$
in $\cnk,$ is a scalar multiple of the orthogonal projection map
$\Q_{\lambda}:\cnk\to\mathcal{U}_{\lambda^{+}(n)}.$ Comparing traces
in $\cnk$ shows that $d_{\lambda^{+}(n)}d_{\lambda}\tilde{\mathcal{Q}}_{\lambda}=\mathcal{Q}_{\lambda}.$
This is summarised below.
\begin{thm}
\label{thm:orthogonal projection is multiple of integral of z}For
any $\lambda\vdash k,$ the orthogonal projection 
\[
\mathcal{Q}_{\lambda}:\cnk\to\mathcal{U}_{\lambda^{+}(n)}
\]
 is equal to $d_{\lambda^{+}(n)}d_{\lambda}\tilde{\mathcal{Q}}_{\lambda},$
where 
\[
\tilde{\mathcal{Q}}_{\lambda}=\int_{S_{n}}g\check{\xi}_{\lambda}^{\mathrm{norm}}\otimes g\normxilambda dg.
\]
 
\end{thm}

\subsection{Evaluating the Projection\label{subsec:CALCULATING THE PROJECTION}}

It remains only to compute 
\[
\tilde{\mathcal{Q}}_{\lambda}=\int_{S_{n}}g\check{\xi}_{\lambda}^{\mathrm{norm}}\otimes g\normxilambda dg
\]
using the Weingarten calculus for the symmetric group (\S\ref{SUBsubsec:THE WEINGARTEN CALCULUS FOR THE SYMMETRIC GROUP})
and evaluate this as an endomorphism of $\cnk.$ Let $z=\normxilambda\otimes\normxilambda$
and denote, for each $g\in S_{n},$ $gz=\rho(g)\normxilambda\otimes\rho(g)\normxilambda,$
so that $\tilde{\mathcal{Q}}_{\lambda}$ and 
\[
\int_{S_{n}}gzdg
\]
define the same element of $\mathrm{End}\left(\cnk\right)$ via the
canonical isomorphisms. We begin by computing, for fixed $\sigma,\sigma'\in S_{k}$
and $I=(i_{1},\dots,i_{2k})$, the coefficient of $e_{I}$ in 
\begin{equation}
\int_{S_{n}}g\left[\vsigmas\right]dg.\label{eq: inner product of projection of z and general e_I}
\end{equation}
Using Proposition \ref{thm:weingartenstrict}, the coefficient of
$e_{I}$ in (\ref{eq: inner product of projection of z and general e_I})
is 
\begin{equation}
\sum_{\pi\in\mathrm{Part}\left([2k]\right)}\pi^{\mathrm{strict}}\left(\vsigmas\right)\pi^{\mathrm{strict}}\left(e_{I}\right)\frac{1}{(n)_{|\pi|}}\label{eq:projection of v sigmas NOT explicit}
\end{equation}
and we evaluate these coefficients in the following sections.

\subsubsection{Conditions for $\pi^{\mathrm{strict}}$ to be non--zero\label{subsec:CONDITIONS FOR =00005Cpi^strict TO BE NON-ZERO}}

Obviously, if $i\neq j,$ we cannot have $v_{\sigma(i)}=v_{\sigma(j)}$
or $v_{\sigma'(i)}=v_{\sigma'(j)}$. So, if $\pi\in\part$ has any
two elements in $[k]$ in the same subset, or any two elements in
$\{k+1,\dots,2k\}$ in the same subset, then 
\[
\pi^{\mathrm{strict}}\left(\vsigmas\right)=0.
\]
This leaves only $\pi\in\part$ in which every element $i\in[k]$
is in a subset with \emph{exactly one }element $j\in\{k+1,\dots,2k\}$,
or is a singleton. Similarly, the elements $j\in[k+1,2k]$ may only
be in a subset with exactly one element $i\in[k]$ or they are singletons.
That is, $\pi\leq\iota(\tau)$ for some $\tau\in S_{k}.$ The only
possible candidate for $\tau$ is detailed in the next lemma. 
\begin{lem}
\label{lem:conditions for pi strict to be non zero}Given $\sigma,\sigma'\in S_{k}$
and $\pi\in\part,$ let $\tau=\sigma^{-1}\sigma'\in S_{k}.$ Then
\[
\pi^{\mathrm{strict}}\left(\vsigmas\right)=0
\]
 unless 
\begin{enumerate}
\item $\pi=\iota(\tau)$, or
\item $\pi\leq\iota(\tau)$.
\end{enumerate}
\end{lem}

\begin{proof}
Fix any $\sigma,\sigma'\in S_{k}.$ Recall that $\iota(\tau)$ is
the partition $\Big\{\{1,k+\tau^{-1}(1)\},\dots,\{k,k+\tau^{-1}(k)\}\Big\}.$
Suppose, towards a contradiction, that $\pi\neq\iota(\tau)$ and $\pi\nleq\iota(\tau)$,
but that 
\[
\pi^{\mathrm{strict}}\left(\vsigmas\right)\neq0.
\]
Then, by definition, there exists a subset $S$ in $\pi$ that is
not contained inside any of the subsets in $\iota(\tau).$ Suppose
$S$ contains $i\in[k]$ and $k+j,$ where $j\in[k]\backslash\{\tau^{-1}(i)\}$
as well (by the previous discussion, if $S$ contains another element
in $[k],$ then 
\[
\pi^{\mathrm{strict}}\left(\vsigmas\right)=0
\]
and, if $S$ contains only the element $i$ or the element $k+\tau^{-1}(i),$
then $S$ would be contained inside one of the subsets in $\iota(\tau)$).
Then, for any $L=(l_{1},\dots,l_{2k}),$ 
\[
\pi^{\mathrm{strict}}\left(e_{L}\right)\neq0\implies l_{i}=l_{k+j}.
\]
 Observe that 
\[
\begin{aligned} & v_{\sigma(1)}\otimes\dots\otimes v_{\sigma(k)}\otimes v_{\sigma'(1)}\otimes\dots\otimes v_{\sigma'(k)}\\
= & \ v_{\sigma(1)}\otimes\dots\otimes e_{2\sigma(i)-1}-e_{2\sigma(i)}\otimes\dots\otimes v_{\sigma(k)}\\
 & \ \otimes v_{\sigma'(1)}\otimes\dots\otimes e_{2\sigma'(j)-1}-e_{2\sigma'(j)}\otimes\dots\otimes v_{\sigma'(k)}\\
= & \ v_{\sigma(1)}\otimes\dots\otimes e_{2\sigma(i)-1}\otimes\dots\otimes v_{\sigma(k)}\otimes v_{\sigma'(1)}\otimes\dots\otimes e_{2\sigma'(j)-1}\otimes\dots\otimes v_{\sigma'(k)}\\
+ & \ v_{\sigma(1)}\otimes\dots\otimes e_{2\sigma(i)}\otimes\dots\otimes v_{\sigma(k)}\otimes v_{\sigma'(1)}\otimes\dots\otimes e_{2\sigma'(j)}\otimes\dots\otimes v_{\sigma'(k)}\\
- & \ v_{\sigma(1)}\otimes\dots\otimes e_{2\sigma(i)}\otimes\dots\otimes v_{\sigma(k)}\otimes v_{\sigma'(1)}\otimes\dots\otimes e_{2\sigma'(j)-1}\otimes\dots\otimes v_{\sigma'(k)}\\
- & \ v_{\sigma(1)}\otimes\dots\otimes e_{2\sigma(i)-1}\otimes\dots\otimes v_{\sigma(k)}\otimes v_{\sigma'(1)}\otimes\dots\otimes e_{\sigma'(j)}\otimes\dots\otimes v_{\sigma'(k)}.
\end{aligned}
\]
So, if $\pi^{\mathrm{strict}}\left(\vsigmas\right)\neq0,$ then either
$\sigma(i)=\sigma'(j),$ which implies $(\sigma')^{-1}\sigma(i)=\tau^{-1}(i)=j$,
a contradiction \textbf{or} $2\sigma(i)=2\sigma'(j)-1$ \emph{or}
$2\sigma(i)-1=2\sigma'(j)$, which is obviously not possible. So,
in any case, we have a contradiction.
\end{proof}
Refinements of $\iota(\tau)$ are obtained by either splitting each
subset $\{i,k+\tau^{-1}(i)\}$ into two subsets $\{i\}$ and $\{k+\tau^{-1}(i)\}$,
or leaving them as they are, which corresponds to `deleting' edges
from the diagram, and $\pi^{\mathrm{strict}}\left(\vsigmas\right)$
can be evaluated in terms of the number of edges deleted.
\begin{rem*}
If $\pi$ is obtained from $\iota(\tau)$ by deleting $p$ edges,
then $|\pi|=|\iota(\tau)|+p=k+p$. 
\end{rem*}

\subsubsection{Evaluating $\pi^{\mathrm{strict}}$ }
\begin{lem}
\label{lem:pi strict of v sigma} Let $\sigma,\sigma'\in S_{k}$ and
let $\tau=\sigma^{-1}\sigma'$. Then, if $\pi=\iota(\tau)$ or $\pi\leq\iota(\tau),$
\[
\pi^{\mathrm{strict}}\left(\vsigmas\right)=(-1)^{|\pi|-k}2^{k}.
\]
\end{lem}

\begin{proof}
We will prove this by induction on the number of edges deleted from
$\iota(\tau)$. The base case is where $\pi=\iota(\tau).$ In this
case, $|\pi|=k.$ We can write $\vsigmas$ as a sum of $2^{2k}$ standard
basis vectors $e_{I}\otimes e_{J}$, with coefficients either $1$
or $-1$. In this case, $\pi^{\mathrm{strict}}(e_{I}\otimes e_{J})=1$
if and only if, for each $i\in[k],$ the $i^{\mathrm{th}}$ coordinate
in $I$ matches the $\tau^{-1}(i)$ coordinate in $J$ exactly. If
this is the case, the coefficient of $e_{I}\otimes e_{J}$ is $1$
and there are $2^{k}$ possible pairings, since for each of the $2^{k}$
possible $I,$ there is exactly one $J$ such that $\pi^{\mathrm{strict}}(e_{I}\otimes e_{J})=1.$ 

Suppose that the claim is true for all $\pi\leq\iota(\tau)$ of size
$\leq p$ and let $\pi\leq\iota(\tau)$ have $|\pi|=p+1.$ Add an
edge $e$ from a singleton $j$ to $k+\tau^{-1}(j)$ to obtain $\pi',$
so that $\pi'\leq\iota(\tau)$ and $|\pi'|=p$, implying that 
\[
\left(\pi'\right)^{\mathrm{strict}}\left(\vsigmas\right)=(-1)^{p-k}2^{k}.
\]
For every $(I,J)$ such that $\text{\ensuremath{\left(\pi'\right)^{\mathrm{strict}}}}\left(e_{I}\otimes e_{J}\right)=1,$
there is exactly one $(I,J^{-})$ such that $\pi^{\mathrm{strict}}\left(e_{I}\otimes e_{J^{-}}\right)=1,$
obtained by swapping the $(\tau^{-1}(j))^{\mathrm{th}}$ coordinate
in $J^{-}$ so that it no longer matches the $j^{\mathrm{th}}$ coordinate
in $I.$ The coefficient of $e_{I}\otimes e_{J^{-}}$ in $\vsigmas$
is $-1$ multiplied by the coefficient of $e_{I}\otimes e_{J}$, so
that 
\[
\pi^{\mathrm{strict}}\left(\vsigmas\right)=-(-1)^{p-k}2^{k}.
\]
 
\end{proof}

\subsubsection{Evaluating the projection}

Lemmas \ref{lem:conditions for pi strict to be non zero} and \ref{lem:pi strict of v sigma}
and (\ref{eq:projection of v sigmas NOT explicit}) imply that, for
each fixed $\sigma,\sigma'\in S_{k},$ 
\[
\begin{aligned} & \int_{S_{n}}g\left(\vsigmas\right)dg\\
= & \sum_{I}\left[\sum_{\pi\leq\iota(\tau)}\pi^{\mathrm{strict}}\left(e_{I}\right)\frac{1}{(n)_{|\pi|}}(-1)^{|\pi|-k}2^{k}\right]e_{I},
\end{aligned}
\]
where the sum is over all multi-indices $I$ that correspond to a
basis vector of $\cntwok$, and $\tau=\sigma^{-1}\sigma'$. Since
\[
z=\frac{1}{2^{k}k!}\sum_{\sigma,\sigma'\in S_{k}}\chi^{\lambda}(\sigma)\chi^{\lambda}(\sigma')\left[v_{\sigma(1)}\otimes\dots\otimes v_{\sigma(k)}\otimes v_{\sigma'(1)}\otimes\dots\otimes v_{\sigma'(k)}\right],
\]
we have 
\begin{equation}
\int_{S_{n}}gzdg=\frac{(-1)^{-k}}{k!}\sum_{\sigma,\sigma'\in S_{k}}\chi^{\lambda}(\sigma)\chi^{\lambda}(\sigma')\sum_{\pi\leq\iota\left(\sigma^{-1}\sigma'\right)}(-1)^{|\pi|}\frac{1}{(n)_{|\pi|}}\sum_{I}\pi^{\mathrm{strict}}\left(e_{I}\right)e_{I}.\label{eq: projection to Sn invariants as element of cntwok strict}
\end{equation}

\subsubsection{Proof of Main Theorem\label{subsec:MAIN THEOREM PROOF}}
\begin{proof}[Proof of Theorem \ref{thm: MAIN THEOREM}]
We take 
\[
\mathcal{U}_{\lambda^{+}(n)}=\Big\langle\Delta(g,\sigma)\left(\xi_{\lambda}\right):\ g\in S_{n},\ \sigma\in S_{k}\Big\rangle_{\C}\cong V^{\lambda^{+}(n)}\otimes V^{\lambda}.
\]
By Theorem \ref{thm:orthogonal projection is multiple of integral of z},
the orthogonal projection $\mathcal{Q}_{\lambda,n}$ is the image
of $d_{\lambda}d_{\lambda^{+}(n)}\tilde{\mathcal{Q}}_{\lambda,n}$
inside $\mathrm{End}\left(\cnk\right)$. Combining (\ref{eq: projection to Sn invariants as element of cntwok strict})
with the character orthogonality formula 
\[
\frac{1}{k!}\sum_{\sigma\in S_{k}}\chi^{\lambda}(\sigma)\chi^{\lambda}(\sigma\tau)=\frac{\chi^{\lambda}(\tau)}{d_{\lambda}}
\]
and Remark \ref{rem: P_pi strict definition}, one sees that this
is exactly (\ref{eq: strict orthogonal projection formula}).
\end{proof}
\bibliographystyle{alpha}
\bibliography{library}

\noindent Ewan Cassidy, Department of Mathematical Sciences, Durham
University, Lower Mountjoy, DH1 3LE, Durham, United Kingdom

\noindent ewan.g.cassidy@durham.ac.uk

\noindent (Current institution: Department of Pure Mathematics and
Mathematical Statistics, University of Cambridge, Wilberforce Road,
Cambridge, CB3 0WB

\noindent email: egc45@cam.ac.uk
\end{document}